\documentclass[12pt]{article}
\usepackage{epsfig}
\usepackage{amsfonts}
\usepackage{amssymb}
\usepackage{amsmath}
\usepackage{amsthm}
\usepackage{latexsym}
\usepackage{graphicx}
\usepackage{bm}
\usepackage{tabularx}
\usepackage{booktabs} 
\usepackage{enumerate}
\usepackage{caption}
\usepackage[titletoc]{appendix}
\usepackage[numbers]{natbib}
\usepackage{bbm}
\usepackage{url}
\usepackage{hyperref}
\usepackage{color}
\usepackage{xcolor}
\usepackage[section]{placeins}
\usepackage{subcaption}
\usepackage{subfloat}
\usepackage{placeins}

\usepackage{floatrow}
\newfloatcommand{capbtabbox}{table}[][\FBwidth]

\usepackage{blindtext}

\DeclareUnicodeCharacter{0308}{ü}

\let\Oldsection\section
\renewcommand{\section}{\FloatBarrier\Oldsection}

\let\Oldsubsection\subsection
\renewcommand{\subsection}{\FloatBarrier\Oldsubsection}

\let\Oldsubsubsection\subsubsection
\renewcommand{\subsubsection}{\FloatBarrier\Oldsubsubsection}

\catcode`\@=11

\newcommand{\shortdot}[1]{\raisebox{-0.4pt}{$\stackrel{\bullet}{#1}$}}

\theoremstyle{plain}
\newtheorem{theorem}{Theorem}[section]
\newtheorem{lemma}[theorem]{Lemma}

\theoremstyle{definition}

\theoremstyle{remark}

\begin{document}
\title{Multi-Delay Differential Equations: A Taylor Expansion Approach}
\author{ Philip Doldo \\ Center for Applied Mathematics \\ Cornell University
\\ 657 Rhodes Hall, Ithaca, NY 14853 \\  pmd93@cornell.edu  \\ 
\and
  Jamol Pender \\ School of Operations Research and Information Engineering \\ Center for Applied Mathematics \\ Cornell University
\\ 228 Rhodes Hall, Ithaca, NY 14853 \\  jjp274@cornell.edu 
 }    

\maketitle

\begin{abstract} It is already well-understood that many delay differential equations with only a single constant delay exhibit a change in stability according to the value of the delay in relation to a critical delay value. Finding a formula for the critical delay is important to understanding the dynamics of delayed systems and is often simple to obtain when the system only has a single constant delay. However, if we consider a system with multiple constant delays, there is no known way to obtain such a formula that determines for what values of the delays a change in stability occurs. In this paper, we present some single-delay approximations to a multi-delay system obtained via a Taylor expansion as well as formulas for their critical delays which are used to approximate where the change in stability occurs in the multi-delay system. We determine when our approximations perform well and we give extra analytical and numerical attention to the two-delay and three-delay settings.

\end{abstract}


\section{Introduction} \label{sec_intro}


Delay differential equations (DDEs) are useful for modeling phenomena that arise in areas such as biology, machine learning, neural networks, physics, and surprisingly queueing theory.  See for example \cite{cooke1996analysis, wei1999stability, freedman1986stability, campbell2007time, wu2011introduction, novitzky2019nonlinear, novitzky2020queues, penkovsky2019coupled, rackauckas2020universal}.  When trying to understand the dynamics of a DDE with only a single constant delay, it is often crucial to find the critical delay at which solutions to the system exhibit a change in stability when treating the delay as a bifurcation parameter. In many cases, deriving an expression for the critical delay in a single-delay system and one space dimension is quite easy and provides valuable information about the behavior of the system. However, DDEs with multiple constant delays are significantly more difficult to analyze than DDEs with only a single constant delay due to the presence of additional exponential terms in the characteristic equation. This complicates the standard procedure typically used for finding the critical delay in single-delay systems and consequently there are far fewer analytical results for multi-delay systems than there are for single-delay systems. When the space is also multi-dimensional, this only complicates the problem even further.  

Despite the complexity to obtain analytic results for multi-delay differential equations, some analytic results do exist for certain classes of DDE systems that have two or three delays. For example, \cite{nussbaum1978differential} proves the existence of periodic solutions for a specific class of DDEs with two delays. Some stability results for DDEs with two delays can be found in \cite{wei1999stability, freedman1986stability, ruan2003zeros, belair1994stability, campbell1999qualitative, shayer2000stability, campbell2006multistability, guo2008two}. In \cite{wei1999stability}, a neural network model with two delays is examined by analyzing its transcendental characteristic equation and by applying normal form theory and the center manifold theorem. In \cite{campbell1999qualitative}, a neural network model where each neuron receives two time delayed inputs is considered and sufficient conditions for stability of an equilibrium are determined by analyzing a system of transcendental equations obtained from the characteristic equation. In \cite{freedman1986stability}, a system of two DDEs, each having two delays, is considered and the Nyquist criterion is applied to the characteristic equation to obtain estimates on the values of delays for which the system's stability is preserved, in addition to other stability criteria. In \cite{ruan2003zeros}, results on the zeros of general transcendental functions are obtained and applied to exponential polynomials and DDEs with two delays.  Additional results for multi-delay systems can be found in \cite{lucken2015classification, llibre2006periodic, yoneyama1988stability} and references therein. In \cite{bernard2001sufficient}, DDEs with distributed delays are considered and analyzed using Laplace transforms.  Regardless of existing results, there is no known general formula for the critical values of the delays in multi-delay systems that determine where the system changes stability and thus a better understanding of the stability properties of multi-delay DDEs is desired.  Our work differs from the existing literature as we focus on using single-delay differential equations, which already have well-understood stability properties, to approximate multi-delay differential equations to better understand the stability of systems with multiple delays.  

Although there are many application areas for DDEs, our work is motivated by multi-delay DDE models from queueing theory.  In recent years, there has been a great interest in studying queues with delayed information, see for example \cite{novitzky2019nonlinear, pender2017queues, novitzky2020limiting, novitzky2020queues, doldo2020breaking, liu2012many, atar2020heavy}.  Queues with delayed information is an important area as they describe how information is lagged in information-technology driven systems.  In some empirical research, it has been observed that delayed information can cause oscillations in healthcare, amusement parks, transportation systems, and even produce prices \cite{tao2017stochastic, dong2019impact, nirenberg2018impact, MACKEY1989497}.


In this paper, we aim to contribute to the understanding of where the change in stability occurs in multi-delay systems by introducing single-delay approximations to a multi-delay system. The main idea is that we can easily compute the critical delays of our single-delay approximations due to their comparatively simple characteristic equations and then use these critical delays to approximate where the multi-delay system undergoes a change in stability. We will primarily focus our attention on two single-delay approximations, one of which is a first-order neutral DDE and one which also contains a delayed second-derivative term. We analyze each of these approximations and numerically examine their performance in both a two-delay system and three-delay system.

 \subsection{Main Contributions of Paper}

The contributions of this work can be summarized as follows:    
\begin{itemize}
\item We develop a neutral DDE with only a single constant delay and derive its critical delay in an attempt to approximate where the change in stability occurs in the multi-delay system.
\item We also present a DDE with a single constant delay that has a delayed second-derivative term and we derive its critical delay and use it to approximate where the change in stability occurs in the multi-delay system.
\item We analyze the performance of these approximations in a two-delay system.
\end{itemize} 


\subsection{Organization of Paper}

The remainder of this paper is organized as follows.   In Section \ref{sec_UD} we discuss an application of multi-delay systems in queueing theory that motivates the analysis in this paper. In Section \ref{Section3} we present two single-delay approximations to a multi-delay system (in addition to discussing the constant-delay approximation), compute their critical delays, and examine their performance numerically in the two-delay and three-delay cases. In Section \ref{section_2_delay} we analyze under which conditions our single-delay approximations perform the best in the two-delay setting. In Section \ref{Section5} we provide several numerical examples comparing our approximations of where the change in stability occurs to where it actually occurs in the multi-delay system as well as a brief comparison of the stability of two-delay systems and three-delay systems. In Section \ref{lambertwsection}, we briefly discuss how our single-delay approximations relate to generalizations of the Lambert W function. Finally, in Section \ref{conclusion}, we give some closing thoughts and discuss potential directions for future research.



\section{Motivating Application}  \label{sec_UD}

In many queueing systems, customers are often supplied with information about queue lengths or waiting times. The information that a customer is provided with can influence the customer's decision of whether or not to join a queue and wait for a service. However, the information that customers are given is often inherently delayed in some sense. For example, waiting times or queue length information is often not updated in real time, so the information that the customer is receiving is actually information about the state of the queueing system in the past. Another way that this information can be delayed is if a customer has to commit to joining a queue before being able to physically travel to join the queue. Indeed, physically traveling to join a queue takes time, so by the time the customer actually arrives at the queue, the information that informed the customer's decision to join the queue is information from the past and the queue length or waiting time could be different when the customer arrives.

As one might expect, the delay introduced from traveling will depend on how far the customer had to physically travel to arrive at the queue. In many situations, not all customers are going to be traveling to the queue from the same location. For example, some service might be available in only a single city and thus this service could receive customers that traveled from multiple nearby cities, each city having a different distance from the destination. Thus, if some percentage of customers must travel from a given location, we can assume that they all experience the same delay. Keeping this in mind, in this section we will introduce a stochastic queueing model where customers travel from location $k$ with probability $p_k$ and experience constant delay $\Delta_k$, as illustrated in Figure \ref{m_locations}. A stochastic queueing model that models customer choice to depend on the queue lengths with a single constant delay has already been introduced in \cite{pender2020stochastic}. The model that we will present is a multi-delay generalization of this stochastic queueing model. Indeed, we consider a system of $N$ infinite-server queues where the choice model used to determine which queue a customer joins is based on a Multinomial Logit Model (MNL) so that the probability of joining the $i^{\text{th}}$ queue is given by the following expression 

\begin{eqnarray}
p_i( Q(t), \Delta ) &=& \frac{\exp(- \theta \left( \sum^{m}_{k=1} p_k Q_i(t-\Delta_k) \right) )}{\sum^{N}_{j=1} \exp(- \theta \left( \sum^{m}_{k=1} p_k Q_j(t-\Delta_k) \right))} 
\end{eqnarray}
where $Q(t) = ( Q_1(t), Q_2(t), ... , Q_N(t))$ and $\Delta = ( \Delta_1, \Delta_2, ... , \Delta_m)$.  


Using these probabilities for joining each queue allows us to construct the following stochastic model for the queue length process of our N dimensional system for $t \geq 0$
\begin{eqnarray}
Q_i(t) &=& Q_i(0) +  \Pi^a_{i} \left( \int^{t}_{0} \frac{\lambda \cdot \exp(- \theta \left( \sum^{m}_{k=1} p_k Q_i(s-\Delta_k) \right) )}{\sum^{N}_{j=1} \exp(- \theta \left( \sum^{m}_{k=1} p_k Q_j(s-\Delta_k) \right) ) } ds \right) \nonumber \\
&&-  \Pi^d_{i} \left( \int^{t}_{0} \mu Q_i(s) ds \right) \label{cdsqnoeta}
\end{eqnarray}
 where each $\Pi(\cdot)$ is a unit rate Poisson process and  $Q_i(s) = \varphi_i(s)$ for all $s \in [- \max_{1 \leq k \leq m}{\Delta_k},0]$.  In this model, for the $i^{th}$ queue, we have that 
  \begin{equation}
 \Pi^a_{i} \left( \int^{t}_{0} \frac{\lambda \cdot \exp(- \theta \left( \sum^{m}_{k=1} p_k Q_i(s-\Delta_k) \right) )}{\sum^{N}_{j=1} \exp(- \theta \left( \sum^{m}_{k=1} p_k Q_j(s-\Delta_k) \right)  ) } ds \right)
 \end{equation}
  counts the number of customers that decide to join the $i^{th}$ queue in the time interval $(0,t]$.  Note that the rate depends on the queue length at times $t - \Delta_k, k=1, ..., m$ and not time $t$, hence representing the lags in information corresponding to each of the $m$ locations that a customer can potentially travel to the queue from.  Similarly, the randomly time changed Poisson process 
  
  \begin{eqnarray}
  \Pi^d_{i} \left( \int^{t}_{0} \mu Q_i(s) ds \right)
  \end{eqnarray}
  
\noindent counts the number of customers that depart the $i^{th}$ queue having received service from an agent or server in the time interval $(0,t]$. However, in contrast to the arrival process, the service process depends on the current queue length and not the past queue length.  

\begin{figure}[!hb]
\centering
\includegraphics[scale=0.9]{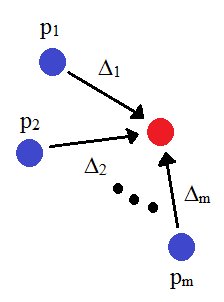}
\caption{Customers travel from one of $m$ locations (represented by \textcolor{blue}{blue} circles) to arrive at the queue (represented by a \textcolor{red}{red} circle) and experience a different constant delay in information depending on this location. Sampling over customers from all locations, we assume that a given customer travels from location $k$ with probability $p_k$ and this customer experiences constant delay $\Delta_k$.}
\label{m_locations}
\end{figure}
 
 \subsection{Fluid Limit  Scaling and Convergence}
 In many service systems, the arrival rate of customers is high.  For example in Disneyland there are thousands of customers moving around the park and deciding on which ride they should join.  Motivated by the large number of customers, we introduce the following scaled queue length process by a parameter $\eta$ 
 \begin{eqnarray}
Q^{\eta}_i(t) &=& Q^{\eta}_i(0) +  \frac{1}{\eta}\Pi^a_{i} \left( \eta \int^{t}_{0} \frac{  \lambda \cdot \exp(- \theta \left( \sum^{m}_{k=1} p_k Q^{\eta}_i(s-\Delta_k) \right))}{ \sum^{N}_{j=1} \exp(- \theta \left( \sum^{m}_{k=1} p_k Q^{\eta}_j(s-\Delta_k) \right) ) } ds \right) \nonumber \\
&&-  \frac{1}{\eta}\Pi^d_{i} \left(\eta \int^{t}_{0} \mu Q^{\eta}_i(s) ds \right). \label{cdsqeta}
\end{eqnarray}

Note that we scale the rates of both Poisson processes, which is different from the many server scaling, which would only scale the arrival rate.  Scaling only the arrival rate would yield a different limit than the one analyzed by \cite{pender2018analysis} since the multinomial logit function is not a homogeneous function.  Letting the scaling parameter $\eta$ go to infinity gives us our first result.

\begin{theorem}\label{fluidlimit}
If $Q_i^{\eta}(s) \to \varphi_i(s)$ almost surely for all $s \in [- \max_k{\Delta_k}, 0]$ and for all $1 \leq i \leq N$, then the sequence of stochastic processes $\{ Q^{\eta}(t) = (Q^{\eta}_1(t),Q^{\eta}_2(t), ..., Q^{\eta}_N(t)  \}_{\eta \in \mathbb{N}}$ converges almost surely and uniformly on compact sets of time to $(q(t) = (q_1(t),q_2(t), ... , q_N(t))$ where
\begin{eqnarray}
\shortdot{q}_i(t) &=& \lambda \cdot \frac{\exp(- \theta \left( \sum^{m}_{k=1} p_k q_i(t-\Delta_k) \right)  )}{\sum^{N}_{j=1} \exp(- \theta \left( \sum^{m}_{k=1} p_k q_j(t-\Delta_k) \right) ) } - \mu q_i(t) \label{ddecd1} 
\end{eqnarray}
and $q_i(s) = \varphi_i(s)$ for all $s \in [- \max_{k} (\Delta_k) ,0]$ and for all  $1 \leq i \leq N$.

\begin{proof}
See Appendix.

\end{proof}
\end{theorem}

This result states that as we let $\eta$ go towards infinity, the sequence of queueing processes converges to a system of \textbf{delay differential equations}.


\subsection{Fluid Limit}
 
As discussed earlier, an interesting extension of the queueing model from previous work is to make the customer choice model depend on delayed information which has multiple constant delays. A physical motivation for the following model is that customers maybe be traveling to the queueing system from multiple different locations which have different travel times to the queueing system. We suppose that customers are in location $k$ with probability $p_k$ and experience delay $\Delta_k > 0$, for $k=1,..., m$.

\begin{align}
\overset{\bullet}{q}_i(t) &= \lambda \frac{\exp \left( - \theta \left[ \sum^m_{k=1} p_k q_i(t - \Delta_k)   \right]  \right)  }{ \sum_{j=1}^N \exp \left( - \theta \left[  \sum^m_{k=1} p_k q_j(t - \Delta_k)   \right]  \right)} - \mu q_i(t), \hspace{5mm} i = 1, ..., N 
\end{align}

   
\noindent It is easy to check that this system has an equilibrium at $q_1 = \cdots  = q_N = \frac{\lambda}{N \mu}$. We can linearize this system about this equilibrium point by introducing the variables $$q_i(t) = \frac{\lambda}{N \mu} + u_i(t), \hspace{5mm} i = 1, ..., N.$$ After making the change of variables, the $i$-th equation is in the form 

\begin{align}
\overset{\bullet}{u}_i(t) &=  \lambda \frac{\exp \left( - \theta \left[ \sum^m_{k=1} p_k u_i(t - \Delta_k)   \right]  \right)  }{ \sum_{j=1}^N \exp \left( - \theta \left[  \sum^m_{k=1} p_k u_j(t - \Delta_k)   \right]  \right)} - \mu u_i(t) - \frac{\lambda}{N}
\end{align}

\noindent and linearizing this system gives us

\begin{align}
\overset{\bullet}{u}_i(t) = \frac{\lambda \theta}{N^2} \sum_{j=1}^{N} \sum_{k=1}^{m} u_j(t - \Delta_k) - \frac{\lambda \theta}{N} \sum_{k=1}^{m} u_i(t - \Delta_k) - \mu u_i(t).
\end{align}




\noindent As done similarly in \citet{novitzky2020limiting}, we can uncouple the delayed part of this system to get the following $N$ equations.

\begin{align}
\overset{\bullet}{v}_1(t) + \mu v_1(t) &= 0\\
\overset{\bullet}{v}_2(t) + \frac{\lambda \theta}{N} \sum_{k=1}^{m} p_k v_2(t - \Delta_k) + \mu v_2(t) &= 0 \label{v2_equation}\\
&\vdots \nonumber \\
\overset{\bullet}{v}_N(t) + \frac{\lambda \theta}{N} \sum_{k=1}^{m} p_k v_N(t - \Delta_k) + \mu v_N(t) &= 0
\end{align}


\noindent We notice that the first of these equations is simply an ODE which we can solve to get that $v_1(t) = \tilde{c} \exp(-\mu t)$ for some undetermined constant $\tilde{c}$, so the analysis of the system reduces to solving the remaining $N-1$ equations. The remaining equations are all in the same form, which is a DDE with multiple delays and we dedicate the next section to analyzing equations in this form.

\section{Single-Delay Approximations}
\label{Section3}

In the previous section, we saw that we could reduce the analysis of the multi-delay system of $N$ queues to the analysis of just a single DDE with multiple constant delays in the form 
\begin{eqnarray}
\overset{\bullet}{u}(t) + \frac{\lambda \theta}{N} \sum_{k=1}^{m} p_k u(t - \Delta_k) + \mu u(t) = 0.
\label{equation313}
\end{eqnarray} However, analyzing even just a single DDE with multiple constant delays is challenging in general. Because of this difficulty, in the following sections we will explore ways of approximating multi-delay DDEs with DDEs that are easier to analyze.

To further motivate the importance of analyzing DDEs with multiple constant delays, it is worth noting that the multi-delay DDEs that we are considering can be viewed as DDEs that only depend on a single discrete random variable delay $\Delta$ where $\Delta = \Delta_k$ with probability $p_k$ for $k = 1, ..., m$, so we could rewrite Equation \ref{equation313} as 
\begin{eqnarray}
\overset{\bullet}{u}(t) + \frac{\lambda \theta}{N} \mathbb{E}[ u(t - \Delta)] + \mu u(t) = 0.
\end{eqnarray}
\noindent Naturally, such DDEs can be used to approximate the behavior of DDEs that depend on a single random variable delay $\Delta$ that is distributed according to some \emph{continuous} probability distribution as we can use a discrete distribution to approximate a continuous distribution. For example, if $\Delta$ were distributed according to a continuous uniform distribution over the interval $[a,b]$, then we would have that 
\begin{eqnarray}
\mathbb{E}[u(t - \Delta)] = \frac{1}{b - a} \int_{a}^{b} u(t - s) \text{d}s
\end{eqnarray}
\noindent and we could approximate this expectation with a finite number of constant delays as follows 
\begin{eqnarray}
\mathbb{E}[u(t - \Delta)] \approx \frac{1}{m} \sum_{k = 1}^{m}  u(t - \Delta_k)
\end{eqnarray}
where $\Delta_k = a + \frac{(b-a)}{m} k$, for $k = 1, ..., m$. Thus, a better understanding of DDEs with multiple constant delays can help us to better understand a variety of problems.


\subsection{Constant-Delay Approximation}


Now we consider a DDE with multiple, $m \geq 2$, constant delays in the form 
\begin{eqnarray} 
\overset{\bullet}{u}(t) = \alpha_0 u(t) + \sum^{m}_{j=1} \alpha_j u(t - \Delta_j) 
\label{eqn_multi_alphas}
\end{eqnarray}

\noindent where $\alpha_j = C p_j$ for some $C \neq 0$ and probabilities $p_j \in [0,1]$ for $j = 1, ..., m$ and $\sum_{j=1}^{m} p_j = 1$. We are motivated to consider equations of this form because Equation \ref{v2_equation} is in this form with $\alpha_0 = - \mu$ and $\alpha_j = - \frac{\lambda \theta}{N} p_j$ for $j = 1, ... , m$. Alternatively, we can write Equation \ref{eqn_multi_alphas} as 
\begin{eqnarray} 
\overset{\bullet}{u}(t) = \alpha_0 u(t) + C \cdot \mathbb{E}[u(t - \Delta)]
\end{eqnarray}

\noindent where we are viewing $\Delta$ as a discrete random variable which has value $\Delta_j$ with probability $p_j$ for $j = 1, ..., m$. Unfortunately in the general case, there does not exist an explicit solution for the critical delay $\Delta_{cr}$, which determines what values of $\Delta_1, ..., \Delta_m$ a Hopf bifurcation will occur at.  This leads us to search for ways to approximate where the change in stability occurs.  A natural first step is to compare the multi-delay system to its constant-delay counterpart, that is 
\begin{eqnarray}
\overset{\bullet}{u}(t) = \alpha_0 u(t) + C \cdot u(t - \Delta^*)
\end{eqnarray}
\noindent for some constant $\Delta^* > 0$, so we have essentially replaced the discrete random variable $\Delta$ with a constant delay $\Delta^*$ to get a single-delay approximation of the multi-delay equation. The advantage of making this approximation is that we can easily compute the critical delay of the constant-delay approximation, which we can use as an approximation for where the change in stability occurs in the multi-delay system. Indeed, the critical delay of the constant-delay system is $$\Delta_{\text{cr}}^{\text{constant}} = \frac{\arccos \left( {-\frac{\alpha_0}{C}} \right)  } {\sqrt{C^2 - \alpha_0^2}}. $$

\noindent In Figure \ref{multidelay_constantdelay_scatterplot}, we numerically compare the stability of the multi-delay system for $m=2$, $N \geq 2$, and $p_1 = p_2 = \frac{1}{2}$ to that of the constant-delay system with $\Delta^* = p_1 \Delta_1 + p_2 \Delta_2 = \mathbb{E}[\Delta]$ by numerically integrating for many $(\Delta_1, \Delta_2)$ pairs and seeing if the solutions are stable or unstable at each point $(\Delta_1, \Delta_2)$. We see that the constant-delay system ends up being more unstable than the multi-delay system and the constant-delay system seems to approximate the change in stability best when $\Delta_1 = \Delta_2$. Of course, it makes sense that the approximation is good when $\Delta_1 = \Delta_2$ as the multi-delay DDE reduces to a single-delay DDE in this case which makes our formula for the critical delay exact.

\begin{figure}[ht!]
\hspace{10mm}
\includegraphics[scale=.7]{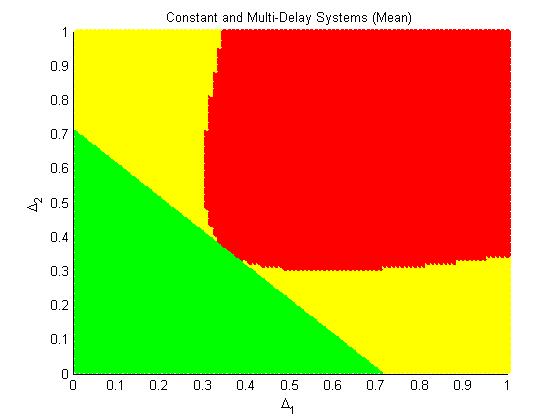}
\captionsetup{justification=centering,margin=2cm}
\caption{\textcolor{green}{Green}: Both the constant-delay and multi-delay systems are stable (solutions converge to an equilibrium) \\ \textcolor{yellow}{Yellow}: constant-delay system is unstable and the multi-delay system is stable \\ \textcolor{red}{Red}: Both the constant-delay and multi-delay systems are unstable (solutions oscillate and approach a limit cycle) \\ $p_1 = p_2 = .5, \alpha_0 = -1, C = -5$}
\label{multidelay_constantdelay_scatterplot}
\end{figure}

The accuracy of the constant-delay approximation will depend on the choice of the constant $\Delta^*$. In the two-delay setting, we numerically compare the performance of the constant-delay system with several different choices of $\Delta^*$ being used.  We created scatterplots such as the one in Figure \ref{multidelay_constantdelay_scatterplot} for several values of $p$ and we considered four different choices of $\Delta^*$. 

\begin{itemize}
\item $\Delta^* = \Delta_1$
\item $\Delta^* = \Delta_2$
\item $\Delta^* = \frac{\Delta_1 + \Delta_2}{2}$ (midpoint)
\item $\Delta^* = p \Delta_1 + (1-p) \Delta_2 = \mathbb{E}[\Delta]$ (mean)
\end{itemize}

\noindent We collected data on how accurately the constant-delay approximation approximates the stability of the multi-delay system under each of these choices of $\Delta^*$ in Table \ref{Table_Accuracy_Constant_Delay}.  For clarity, the accuracy percentages are computed according to the formula $$\text{Accuracy } \% = 100 \cdot \frac{\text{ \# of points with same color in approximation and multi-delay plots}}{\text{total \# of points considered in the unit square}} .$$ It is worth noting that this definition of accuracy is somewhat arbitrary due to the restriction of only considering $(\Delta_1, \Delta_2)$ pairs in the unit square and sampling $(\Delta_1, \Delta_2)$ pairs from other regions would give different results in general. We give some examples of this in the Appendix.

From the information in Table \ref{Table_Accuracy_Constant_Delay}, we observe that the choice $\Delta^* = \mathbb{E}[\Delta]$ gives the best results on average (uniformly sampling over each value of $p$ in the table), though other choices of $\Delta^*$ perform better for some specific values of $p$. Moving forward, we will see if we can improve upon the constant-delay approximation by introducing other approximations of the multi-delay system in the following two sections.

\begin{table}[]
\begin{center}
\begin{tabular}{| l | l | l | l | l | l |}
\hline\\ [-2.5ex]
$p$ & $\Delta^* = \Delta_1$ & $\Delta^* = \Delta_2$  & Midpoint & Mean  \\
\hline
.01  & 53.91\% & 99.99\% & 76.95\% & 99.52\%\\
\hline
.1 & 54.28\% & 97.40\% & 77.32\% & 95.82\%\\
\hline
.2 & 56.58\% & 94.62\% & 79.62\% & 91.22\% \\
\hline
.3 & 62.65\% & 89.15\% & 79.41\% & 83.71\% \\
\hline
.4 &  71.20\% & 81.96\% & 73.88\% & 74.76\% \\
\hline
.5 & 76.96\% & 76.96\% & 72.38\% & 72.38\%\\
\hline
.6 & 81.96\% & 71.20\% & 73.88\% & 74.76\%\\
\hline
.7 & 89.15\% & 62.65\% & 79.41\% & 83.71\% \\
\hline
.8 & 94.62\% & 56.58\% & 79.62\% & 91.22\% \\
\hline
.9 & 97.40\% & 54.28\% & 77.32\% & 95.82\% \\
\hline
.99 & 99.99\% & 53.91\% & 76.95\% & 99.52\%  \\
\hline
\end{tabular}
\end{center}
\caption{The percentage of points in the constant-delay approximation scatterplots that correctly matched those in the multi-delay system scatterplot for various values of $p$, with $C = -5$ and $\alpha_0 = -1$. Based on the table, the average accuracy percentage for the $\Delta^* = \Delta_1$ case is about 76.25\%, for the $\Delta^* = \Delta_2$ case it is about 76.25\%, for the midpoint case it is about $76.98\%$, and for the mean case it is about 87.49\%.  }
\label{Table_Accuracy_Constant_Delay}
\end{table}

\subsection{Neutral Approximation}
\label{subsection31}

The observations in the previous section help motivate the main focus of this work in which we use a novel approximation that exploits the fact that we know how to calculate the critical delay in a neutral DDE with one constant delay.  Our goal is to approximate our multi-delay system with a neutral DDE that has a single delay, which we can do by Taylor expanding, as follows

 
 \begin{eqnarray}
\overset{\bullet}{u}(t) &=& \alpha_0 u(t) + \sum^{m}_{j=1} \alpha_j u(t - \Delta_j) \\
&=& \alpha_0 u(t) + \sum^{m}_{j=1} \alpha_j  u(t - \Delta^* + \Delta^* - \Delta_j) \\
&\approx& \alpha_0 u(t) + \sum^{m}_{j=1} \alpha_j   [ u(t - \Delta^*) + (\Delta^* - \Delta_j) \cdot \overset{\bullet}{u}(t-\Delta^*)] \\
&=& \alpha_0 u(t) + \left( \sum^{m}_{j=1} \alpha_j \right) \cdot u(t - \Delta^*) +  \left( \sum^{m}_{j=1} \alpha_j (\Delta^* - \Delta_j) \right) \cdot \overset{\bullet}{u}(t-\Delta^*)\\
&=& \alpha_0 u(t) + A_0 \cdot u(t - \Delta^*) + A_1 \cdot \overset{\bullet}{u}(t - \Delta^*)
\end{eqnarray}

\noindent where $$A_0 = \sum_{j=1}^{m} \alpha_j \hspace{5mm} \text{and} \hspace{5mm}  A_1 = \sum_{j=1}^{m} \alpha_j (\Delta^* - \Delta_j).$$ It is important to recognize that $A_1$ depends on all of the delays $\Delta_1, ..., \Delta_m$ as well as the choice of $\Delta^*$, which might depend on any of $\Delta_1, ..., \Delta_m$ itself. More specifically, $A_1$ explicitly depends on the differences $\Delta^* - \Delta_j$ for $j = 1, ..., m$. Despite this dependence, we will essentially be treating $A_1$ as a constant parameter in our analysis. With this understanding, we want to obtain an expression for the critical value of $\Delta^*$ corresponding to when solutions to this neutral DDE exhibit a change in stability. The critical delay of this neutral DDE can then be viewed as an approximation for where the change in stability occurs in the multi-delay system in the sense that the expression for the critical delay of the neutral DDE can take $\Delta^*$ and $\Delta_1, ..., \Delta_m$ as input and then be used to, with hopefully a reasonable degree of accuracy, determine the stability of the multi-delay system at the point $(\Delta_1, ..., \Delta_m)$ by checking whether or not $\Delta^*$ is greater than or less than the critical delay of the neutral DDE.

\begin{theorem}
The approximate critical delay under the neutral approximation is given by 
\begin{equation}
\Delta_{cr}^{approx} = \frac{1}{\omega }\arccos \left( \frac{- \alpha_0 A_0 + \omega^2 A_{1}}{A_0^2 + \omega^2 A_{1}^2}  \right)
\end{equation}

\noindent where $$\omega = \sqrt{\frac{A_0^2 - \alpha_0^2}{1 - A_1^2}}$$ and $$A_0 = \sum_{j=1}^{m} \alpha_j, \hspace{5mm} A_1 = \sum_{j=1}^{m} \alpha_j (\Delta^* - \Delta_j).$$

\begin{proof}
We start the proof by starting with the approximation given by the neutral DDE, i.e.
 \begin{eqnarray}
\overset{\bullet}{u}(t) &=& \alpha_0 u(t) + \left( \sum^{m}_{j=1} \alpha_j \right) \cdot u(t - \Delta^*) +  \left( \sum^{m}_{j=1} \alpha_j (\Delta^* - \Delta_j) \right) \cdot \overset{\bullet}{u}(t-\Delta^*)\\
&=& \alpha_0 u(t) + A_0 \cdot u(t - \Delta^*) +  A_1 \cdot \overset{\bullet}{u}(t-\Delta^*).
\end{eqnarray}

\noindent Letting $u(t) = e^{rt}$, we get the following characteristic equation. $$r = \alpha_0 + A_0 e^{-r \Delta^*} + A_1 r e^{-r \Delta^*}$$ A change in stability occurs when the real part of $r$ changes sign, so we consider when $r$ is on the imaginary axis so that $r = i \omega$ for some $\omega \in \mathbb{R}$. Separating real and imaginary parts yields the following system of equations.

\begin{align}
-\alpha_0 &=  A_0 \cos(\omega \Delta^*) + \omega A_1 \sin(\omega \Delta^*) \label{neutral_real_parts}\\
\omega &= - A_0 \sin(\omega \Delta^*) + \omega A_1 \cos(\omega \Delta^*) \label{neutral_imaginary_parts}
\end{align}

\noindent We can solve this system to get 

\begin{align}
\cos(\omega \Delta^*) &= \frac{- \alpha_0 A_0 + \omega^2 A_1}{A_0^2 + \omega^2 A_1^2} \label{cosine_expression_neutral}\\
\sin(\omega \Delta^*) &=  \frac{- \omega A_0 - \alpha_0 \omega A_1}{A_0^2 + \omega^2 A_1^2}
\end{align}

\noindent and using the identity $\cos^2(x) + \sin^2(x) = 1$ on Equations \ref{neutral_real_parts} and \ref{neutral_imaginary_parts}, we get $$\omega = \sqrt{\frac{A_0^2 - \alpha_0^2}{1 - A_1^2}}.$$ Then, from our cosine expression in Equation \ref{cosine_expression_neutral}, we are able to solve for the critical delay for the neutral equation, which is meant to approximate the critical delay of the queueing system. Thus, we have $$\Delta_{\text{cr}}^{approx} = \frac{1}{\omega }\arccos \left( \frac{- \alpha_0 A_0 + \omega^2 A_1}{A_0^2 + \omega^2 A_1^2}  \right).$$

\end{proof}
\end{theorem}

To get an idea of how well the critical delay for the neutral system approximates the change in stability of the multi-delay system, we numerically integrated the multi-delay system with $m=2$, $p_1 = p$, and $p_2 = 1-p$ for many $(\Delta_1, \Delta_2)$ pairs and determined if the solution was stable or unstable at each point $(\Delta_1, \Delta_2)$ to create a scatterplot to show where in the $\Delta_1$-$\Delta_2$ plane the change in stability occurs and then compared this to analogous scatterplots for the neutral approximation for various choices of $\Delta^*$ and checked what percentage of points in each scatterplot matched with those in the scatterplot for the multi-delay system. We created such scatterplots for several values of $p$ and we considered four different choices of $\Delta^*$. 

\begin{itemize}
\item $\Delta^* = \Delta_1$
\item $\Delta^* = \Delta_2$
\item $\Delta^* = \frac{\Delta_1 + \Delta_2}{2}$ (midpoint)
\item $\Delta^* = p \Delta_1 + (1-p) \Delta_2 = \mathbb{E}[\Delta]$ (mean)
\end{itemize}

\noindent Examples of these scatterplots can be found in Section \ref{subsection_51}. We collected data on how accurate each approximation is in Table \ref{Table_Accuracy_Neutral}. We observe that the mean approximation $\Delta^* = \mathbb{E}[\Delta]$ does the best overall, however it does poorly near $p = \frac{1}{2}$. We will give an explanation of this in Section \ref{section_2_delay}. It is worth noting that the definition of accuracy used in Table \ref{Table_Accuracy_Neutral} is based on $(\Delta_1, \Delta_2)$ pairs sampled from the unit square, but if we had instead sampled from a different-sized square then we would in general get different accuracies. Indeed, in the Appendix we consider the scatterplots from sampling $(\Delta_1,\Delta_2)$ pairs from a two-by-two square and a five-by-five square and see that the analogously defined accuracy differs from that in the unit square case.

\begin{table}[]
\begin{center}
\begin{tabular}{| l | l | l | l | l | l |}
\hline\\ [-2.5ex]
$p$ & $\Delta^* = \Delta_1$ & $\Delta^* = \Delta_2$  & Midpoint & Mean  \\
\hline
.01  & 51.32\% & 99.52\% & 76.08\% & 99.52\%\\
\hline
.1 & 52.73\% & 95.14\% & 78.95\% & 95.82\%\\
\hline
.2 & 55.90\% & 86.63\% & 79.46\% & 91.22\% \\
\hline
.3 & 62.39\% & 77.48\% & 77.28\% & 83.71\% \\
\hline
.4 &  70.18\% & 72.09\% & 73.40\% & 74.76\% \\
\hline
.5 & 72.13\% & 72.13\% & 72.38\% & 72.38\%\\
\hline
.6 & 72.09\% & 70.18\% & 73.40\% & 74.76\%\\
\hline
.7 & 77.48\% & 62.39\% & 77.28\% & 83.71\% \\
\hline
.8 & 86.63\% & 55.90\% & 79.46\% & 91.22\% \\
\hline
.9 & 95.14\% & 52.73\% & 78.95\% & 95.82\% \\
\hline
.99 & 99.52\% & 51.32\% & 76.08\% & 99.52\%  \\
\hline
\end{tabular}
\end{center}
\caption{The percentage of points in the neutral approximation scatterplots that correctly matched those in the multi-delay system scatterplot for various values of $p$, with $C = -5$ and $\alpha_0 = -1$. Based on the table, the average accuracy percentage for the $\Delta^* = \Delta_1$ case is about 72.32\%, for the $\Delta^* = \Delta_2$ case it is about 72.32\%, for the midpoint case it is about $76.61\%$, and for the mean case it is about 87.49\%.  }
\label{Table_Accuracy_Neutral}
\end{table}


 \subsection{Keeping the Second Derivative}
\label{subsection_32}

Now that we have looked at various choices of $\Delta^*$ for our neutral approximation, it is natural to question what the best choice of $\Delta^*$ is. We explore this question next. Consider an equation in the form $$\overset{\bullet}{u}(t) = \alpha_0 u(t) + C \left[ \sum_{j=1}^{m} p_j u(t - \Delta_j)  \right] $$ where the discrete random variable $\Delta$ is equal to $\Delta_j$ with probability $p_j$ for $j = 1, ..., m$. By Taylor expanding, we obtain the neutral approximation to this equation.

\begin{align*}
\overset{\bullet}{u}(t) &= \alpha_0 u(t) + C \left[ \sum_{j=1}^{m} p_j u(t - \Delta^* + \Delta^* - \Delta_j)  \right]\\
&= \alpha_0 u(t) + C \left[  \sum_{j=1}^{m} p_j \left( u(t - \Delta^*) + (\Delta^* - \Delta_j) u'(t - \Delta^*) + \frac{1}{2!} (\Delta^* - \Delta_j)^2 u''(t - \Delta^*) + \cdots  \right)  \right] \\
&= \alpha_0 u(t) + C \left[  \sum_{k=0}^{\infty} \frac{1}{k!} \mathbb{E}\left[ (\Delta^* - \Delta)^k  \right] u^{(k)}(t - \Delta^*)  \right]\\
&= \alpha_0 u(t) + C \left[  u(t - \Delta^*) + \mathbb{E}[(\Delta^* - \Delta)] u'(t - \Delta^*)  \right] + \tilde{R}\\
&\approx  \alpha_0 u(t) + C \left[  u(t - \Delta^*) + \mathbb{E}[(\Delta^* - \Delta)] u'(t - \Delta^*)  \right]
\end{align*}

\noindent We want to choose the value of $\Delta^*$ that minimizes the remainder $\tilde{R}$ in some sense. We do this by minimizing the coefficient of the leading-order term in $\tilde{R}$. The leading term in $\tilde{R}$ is $$\frac{1}{2!} \mathbb{E}[(\Delta^* - \Delta )^2] u''(t - \Delta^*)$$ and so we want to minimize $$f(\Delta^*) = \frac{1}{2!}\mathbb{E}[(\Delta^* - \Delta)^2] = \frac{1}{2!} \sum_{j=1}^{m} p_j  (\Delta^* - \Delta_j)^2.$$ Setting the first derivative equal to 0, we get

$$ f'(\Delta^*) =  \sum_{j=1}^{m} p_j (\Delta^* - \Delta_j) =  \sum_{j=1}^{m} p_j \Delta^* -  \sum_{j=1}^{m} p_j \Delta_j =  \Delta^* -  \mathbb{E}[\Delta] = 0$$ and $f''(\Delta^*) > 0$ so we have that $\Delta^* = \mathbb{E}[\Delta]$ minimizes $f(\Delta^*)$, at which point we have $$f(\mathbb{E}[\Delta]) =\frac{1}{2!} \mathbb{E}[(\mathbb{E}[\Delta] - \Delta)^2] = \frac{1}{2!} \mathrm{Var} (\Delta)$$ and higher-order terms in $\tilde{R}$ contain the subsequent central moments. This result supports what we observed numerically as $\Delta^* = \mathbb{E}[\Delta]$ minimizes the leading-order term in the error of the neutral approximation to the multi-delay equation. We remark that making the choice $\Delta^* = \mathbb{E}[\Delta]$ makes the delayed first-derivative term vanish and so the DDE is no longer neutral in this case.

Seeing that our neutral approximation using the mean appears to be the best neutral approximation we can get if we only keep the first derivative, we naturally want to improve upon it given its poor performance near $p = \frac{1}{2}$ in the two-delay case. We aim to improve the neutral approximation from the previous section by Taylor expanding in the same fashion but keeping the second derivative term.

\begin{align*}
\overset{\bullet}{u}(t) &\approx \alpha_0 u(t) + \left( \sum_{j=1}^{m} \alpha_j \right)  u(t - \Delta^*) + \left( \sum_{j=1}^{m} \alpha_j (\Delta^* - \Delta_j)  \right)  \overset{\bullet}{u}(t - \Delta^*)\\
& + \left( \frac{1}{2} \sum_{j=1}^{m} \alpha_j (\Delta^* - \Delta_j)^2 \right)  \overset{\bullet \bullet}{u} (t - \Delta^*)\\
&= \alpha_0 u(t) + A_0 u(t - \Delta^*) + A_1 \overset{\bullet}{u}(t - \Delta^*) + A_2 \overset{\bullet \bullet}{u}(t - \Delta^*)
\end{align*}

\noindent where $$A_0 = \sum_{j=1}^{m} \alpha_j, \hspace{5mm} A_1 = \sum_{j=1}^{m} \alpha_j (\Delta^* - \Delta_j), \hspace{5mm} A_2 = \frac{1}{2} \sum_{j=1}^{m} \alpha_j (\Delta^* - \Delta_j)^2.$$

\noindent As before, we note that $A_2$, in addition to $A_1$, depends on the differences $(\Delta^* - \Delta_j)$ for $j = 1, ..., m$. Just as we did for the neutral approximation, we find the critical delay corresponding to this second-derivative equation with the goal of using it to approximate where the change in stability occurs in the multi-delay system.

\begin{theorem}
If $A_2 \neq 0$, the approximate critical delay under the second-derivative approximation is given by 

$$\Delta_{\text{cr}}^{approx_2} = \frac{1}{\omega}\arccos \left( - \frac{\alpha_0 (A_0 - A_2 \omega^2)^2 + A_1 \omega^2 (A_0 - A_2 \omega^2) + \alpha_0 \omega^2 A_1^2 + A_1^3 \omega^3}{(A_0 - A_2 \omega^2)^3}  \right)$$ where $$\omega^2 = \frac{1 + 2A_0 A_2 - A_1^2}{2 A_2^2} \pm \frac{\sqrt{\left( A_1^2 - 2A_0 A_2 - 1   \right)^2 - 4 A_2^2 \left( A_0^2 - \alpha_0^2 \right)}}{2 A_2^2}.$$
\end{theorem}

\begin{proof}

We begin the proof by starting with the second-derivative approximation

$$\overset{\bullet}{u}(t) = \alpha_0 u(t) +A_0 u(t - \Delta^*) + A_1 \overset{\bullet}{u}(t - \Delta^*) + A_2 \overset{\bullet \bullet}{u}(t - \Delta^*).$$

\noindent Letting $u(t) = e^{rt}$ and $r = i \omega$, we get $$i \omega = \alpha_0 + A_0 e^{- i \omega \Delta^*} + i \omega A_1 e^{- i \omega \Delta^*} - \omega^2 A_2 e^{- i \omega \Delta^*}$$ and separating real and imaginary parts gives us the following system.

\begin{align}
- \alpha_0 &= \cos(\omega \Delta^*)[A_0 - A_2 \omega^2] + \sin(\omega \Delta^*) \cdot A_1 \omega\\
\omega &= - \sin(\omega \Delta^*) [A_0 - A_2 \omega^2] + \cos(\omega \Delta^*) \cdot A_1 \omega
\end{align}

\noindent This implies that $$\alpha_0^2 + \omega^2 = A_0^2 - 2 A_2 A_0 \omega^2 + A_2^2 \omega^4 + A_1^2 \omega^2$$ so that $$A_2^2 \omega^4 + \left( A_1^2 - 2A_0 A_2 - 1  \right)\omega^2 + \left( A_0^2 - \alpha_0^2  \right) = 0$$ which is a quadratic equation in $\omega^2$. Solving this, we get that $$\omega^2 = \frac{1 + 2A_0 A_2 - A_1^2}{2 A_2^2} \pm \frac{\sqrt{\left( A_1^2 - 2A_0 A_2 - 1   \right)^2 - 4 A_2^2 \left( A_0^2 - \alpha_0^2 \right)}}{2 A_2^2}.$$ From the above system, we can get $$\cos(\omega \Delta^*) = - \frac{\alpha_0 (A_0 - A_2 \omega^2)^2 + A_1 \omega^2 (A_0 - A_2 \omega^2) + \alpha_0 \omega^2 A_1^2 + A_1^3 \omega^3}{(A_0 - A_2 \omega^2)^3}$$ and thus $$\Delta_{\text{cr}}^{approx_2} = \frac{1}{\omega}\arccos \left( - \frac{\alpha_0 (A_0 - A_2 \omega^2)^2 + A_1 \omega^2 (A_0 - A_2 \omega^2) + \alpha_0 \omega^2 A_1^2 + A_1^3 \omega^3}{(A_0 - A_2 \omega^2)^3}  \right).$$

\end{proof}

We notice that we have two roots for $\omega^2$ and consequently two values for $\Delta_{\text{cr}}^{approx_2}$. Based on numerical experiments, it seems that the negative root is significantly more accurate than the positive root, so in the numerics that follow we use $$\Delta_{\text{cr}}^{approx_2} = \frac{1}{\omega}\arccos \left( - \frac{\alpha_0 (A_0 - A_2 \omega^2)^2 + A_1 \omega^2 (A_0 - A_2 \omega^2) + \alpha_0 \omega^2 A_1^2 + A_1^3 \omega^3}{(A_0 - A_2 \omega^2)^3}  \right)$$ where $$\omega^2 = \frac{1 + 2A_0 A_2 - A_1^2}{2 A_2^2} - \frac{\sqrt{\left( A_1^2 - 2A_0 A_2 - 1   \right)^2 - 4 A_2^2 \left( A_0^2 - \alpha_0^2 \right)}}{2 A_2^2}.$$

We now consider several choices of $\Delta^*$ in the second-derivative approximation and numerically compare the stability of the approximation under each choice against that of the multi-delay system for $m=2$ and several values of $p$. We consider the following four choices of $\Delta^*$. 

\begin{itemize}
\item $\Delta^* = \Delta_1$
\item $\Delta^* = \Delta_2$
\item $\Delta^* = \frac{\Delta_1 + \Delta_2}{2}$ (midpoint)
\item $\Delta^* = p \Delta_1 + (1-p) \Delta_2 = \mathbb{E}[\Delta]$ (mean)
\end{itemize}

 We collect this information in Table \ref{table_2nd_derivative_delta_star_choices} and it turns out that the choice $\Delta^* = \mathbb{E}[\Delta]$ is by far the best choice according to the table.

\begin{table}[]
\begin{center}
\begin{tabular}{| l | l | l | l | l | l |}
\hline\\ [-2.5ex]
$p$ & $\Delta^* = \Delta_1$ & $\Delta^* = \Delta_2$  & Midpoint & Mean  \\
\hline
.01  & 26.76\% & 99.58\% & 47.18\% & 99.86\%\\
\hline
.1 & 29.28\% & 95.53\% & 53.60\% & 99.40\%\\
\hline
.2 & 31.84\% & 81.20\% & 60.18\% & 96.83\% \\
\hline
.3 & 35.31\% & 62.96\% & 71.35\% & 95.06\% \\
\hline
.4 &  39.69\% & 51.84\% & 88.42\% & 98.47\% \\
\hline
.5 & 44.91\% & 44.91\% & 99.38\% & 99.38\%\\
\hline
.6 & 51.84\% & 39.69\% & 88.42\% & 98.47\%\\
\hline
.7 & 62.96\% & 35.31\% & 71.35\% & 95.06\% \\
\hline
.8 & 81.20\% & 31.84\% & 60.18\% & 96.83\% \\
\hline
.9 & 95.53\% & 29.28\% & 53.60\% & 99.40\% \\
\hline
.99 & 99.58\% & 26.76\% & 47.18\% & 99.86\%  \\
\hline
\end{tabular}
\end{center}
\caption{The percentage of points in the second-derivative scatterplots that correctly matched those in the scatterplot corresponding to the multi-delay system with $m = 2$ for various values of $p$ and choices of $\Delta^*$, with $\alpha_0 = -1$ and $C = -5$. Based on the table, the average accuracy percentage for the $\Delta^* = \Delta_1$ case is about 54.45\%, for the $\Delta^* = \Delta_2$ case it is about 54.45 \%, for the midpoint case it is about $67.35\%$, and for the mean case it is about 98.06\%.   }
\label{table_2nd_derivative_delta_star_choices}
\end{table}


Given that the choice of $\Delta^* = \mathbb{E}[\Delta]$ was the most successful approximation in both the neutral and second-derivative approximations, we now consider this choice of $\Delta^*$ and compare the neutral and second-derivative approximations to the the multi-delay system with $m=2$ for various values of $p$. We collect this information in Table \ref{Table_Accuracy_2nd} and we show examples of the scatterplots we generated while collecting this data in Section \ref{subsection_52}. The second-derivative approximation is $$\overset{\bullet}{u}(t) = \alpha_0 u(t) + C \cdot u(t - \Delta^*)+ C [\Delta^{*} - \mathbb{E}[\Delta] ] \cdot  \overset{\bullet}{u}(t - \Delta^*)+ \frac{C}{2} \mathbb{E}[(\Delta^* - \Delta)^2] \cdot \overset{\bullet \bullet}{u}(t - \Delta^*). $$ However, if we make the choice $\Delta^* = \mathbb{E}[\Delta]$, then this reduces to $$\overset{\bullet}{u}(t) = \alpha_0 u(t) + C \cdot u(t - \mathbb{E}[\Delta]) + \frac{C}{2} \text{Var}(\Delta) \cdot \overset{\bullet \bullet}{u}(t - \Delta^*). $$ It is worth mentioning that with this choice of $\Delta^*$, we have that $\Delta_{\text{cr}}^{approx_2}$ is undefined when $\text{Var}(\Delta) = 0$ as this implies that $A_2 = 0$ and thus the expression for $\omega^2$ has division by zero. When $A_2 = 0$, the critical delay $\Delta_{\text{cr}}^{approx_2}$ of the second-derivative approximation reduces to that of the neutral approximation $\Delta_{\text{cr}}^{approx}$ because the second-derivative term in the DDE vanishes and thus reduces to a neutral equation.

\begin{table}[]
\begin{center}
\begin{tabular}{| l | l | l | l | l | l |}
\hline\\ [-2.5ex]
$p$ & Neutral (Mean) & Second Derivative (Mean) \\
\hline
.01  &  99.52\% & 99.86\%\\
\hline
.1 &  95.82\% & 99.40\% \\
\hline
.2 &  91.22\% & 96.83\%\\
\hline
.3 &  83.71\% & 95.06\%\\
\hline
.4 &   74.76\% &98.47\%\\
\hline
.5 &  72.38\%& 99.38\%\\
\hline
.6 &  74.76\%& 98.47\%\\
\hline
.7 &  83.71\% & 95.06\%\\
\hline
.8 &  91.22\% & 96.83\%\\
\hline
.9 &  95.82\% & 99.40\%\\
\hline
.99 &  99.52\% &99.86\% \\
\hline
\end{tabular}
\end{center}
\caption{The percentage of points in the neutral approximation and second-derivative approximation scatterplots that correctly matched those in the scatterplot corresponding to the multi-delay system with $m = 2$ for various values of $p$, with $\alpha_0 = -1$ and $C = -5$. Based on the table, the average accuracy percentage for the neutral case is about 87.49\% and for the second-derivative case it is about 98.06\%.  }
\label{Table_Accuracy_2nd}
\end{table}


\begin{table}[]
\begin{center}
\begin{tabular}{| l | l | l | l | l | l | l | l | l | l |}
\hline\\ [-2.5ex]
     & $p_1=.1$ & $p_1=.2$  & $p_1=.3$ & $p_1=.4$ & $p_1=.5$ & $p_1=.6$ & $p_1=.7$ & $p_1=.8$ \\
\hline
$p_2=.1$  &         90.96\% &          85.31\% &          75.50\% &           68.86\% &            68.86\% &            75.50\% &            85.31\% &           90.96\% \\
\hline
$p_2=.2$ &          85.31\% &          77.25\% &          68.36\% &           65.01\% &            68.36\% &            77.25\% &            85.31\% &            \\
\hline
$p_2=.3$ &          75.50\% &          68.36\% &          62.72\% &           62.72\% &            68.36\% &            75.50\% &                  &            \\
\hline
$p_2=.4$ &          68.86\% &          65.01\% &          62.72\% &           65.01\% &            68.86\% &                  &                  &            \\
\hline
$p_2=.5$ &          68.86\% &          68.36\% &          68.36\% &          68.86\% &                  &                  &                  &            \\
\hline
$p_2=.6$ &          75.50\% &          77.25\% &          75.50\% &                 &                  &                  &                  &            \\
\hline
$p_2=.7$ &          85.31\% &          85.31\% &                &                 &                  &                  &                  &            \\
\hline
$p_2=.8$ &          90.96\% &                &                &                 &                  &                  &                  &            \\
\hline
\end{tabular}
\end{center}
\caption{The percentage of points in the \textbf{neutral approximation} scatterplot that correctly matched those in the scatterplot corresponding to the multi-delay system with $m = 3$ for various values of $p_1$ and $p_2$ where $p_3 = 1 - p_1 - p_2$, with $\alpha_0 = -1$ and $C = -5$.}
\label{table5}
\end{table}


\begin{table}[]
\begin{center}
\begin{tabular}{| l | l | l | l | l | l | l | l | l | l |}
\hline\\ [-2.5ex]
     & $p_1=.1$ & $p_1=.2$  & $p_1=.3$ & $p_1=.4$ & $p_1=.5$ & $p_1=.6$ & $p_1=.7$ & $p_1=.8$ \\
\hline
$p_2=.1$  &         98.93\% &          95.54\% &          96.13\% &           98.56\% &            98.56\% &            96.13\% &            95.54\% &           98.93\% \\
\hline
$p_2=.2$ &          95.54\% &          94.71\% &          95.54\% &           97.13\% &            95.44\% &            94.71\% &            95.54\% &            \\
\hline
$p_2=.3$ &          96.13\% &          95.44\% &          95.61\% &           95.61\% &            95.44\% &            96.13\% &                  &            \\
\hline
$p_2=.4$ &          98.56\% &          97.13\% &          95.61\% &           97.13\% &            98.56\% &                  &                  &            \\
\hline
$p_2=.5$ &           98.56\% &          95.44\% &          95.44\% &          98.56\% &                  &                  &                  &            \\
\hline
$p_2=.6$ &          96.13\% &          94.71\% &          96.13\% &                 &                  &                  &                  &            \\
\hline
$p_2=.7$ &          95.54\% &          95.54\% &                &                 &                  &                  &                  &            \\
\hline
$p_2=.8$ &          98.93\% &                &                &                 &                  &                  &                  &            \\
\hline
\end{tabular}
\end{center}
\caption{The percentage of points in the \textbf{second-derivative approximation} scatterplot that correctly matched those in the scatterplot corresponding to the multi-delay system with $m = 3$ for various values of $p_1$ and $p_2$ where $p_3 = 1 - p_1 - p_2$, with $\alpha_0 = -1$ and $C = -5$.}
\label{table6}
\end{table}

We see that the second-derivative approximation performs very well in the two-delay case for all values of $p$. However, it is worth noting that it does the worst when $p$ is somewhere between $.2$ and $.4$ or $.6$ and $.8$. We will explore the reasoning behind this in Section \ref{section_2_delay}.

We also considered both the neutral and second-derivative approximations in the three-delay setting, taking points $(\Delta_1, \Delta_2, \Delta_3)$ from the unit cube. In Table \ref{table5} and Table \ref{table6}, we collected data on how accurate the neutral and second-derivative approximations are, respectively, for several values of $p_1$ and $p_2$ (letting $p_3 = 1 - p_1 - p_2$). Indeed, we see that the second-derivative approximation still seems to perform decently well in the three-delay setting. 

When working in a three-delay setting, another potentially interesting choice of $\Delta^*$ to consider is $\Delta^* = \text{median}(\Delta)$. However, choosing $\Delta^*$ to be the mean still performed better numerically in all of the situations that we considered. See, for example, the scatterplots in Figure \ref{median_example} where we have scatterplots for both the neutral and second-derivative approximations, each with both the mean and median choices for $\Delta^*$ considered.

\vspace{5mm}

\begin{figure}[!hb]
     \begin{subfigure}[b]{0.3\textwidth}
         \centering
       \includegraphics[width=\textwidth]{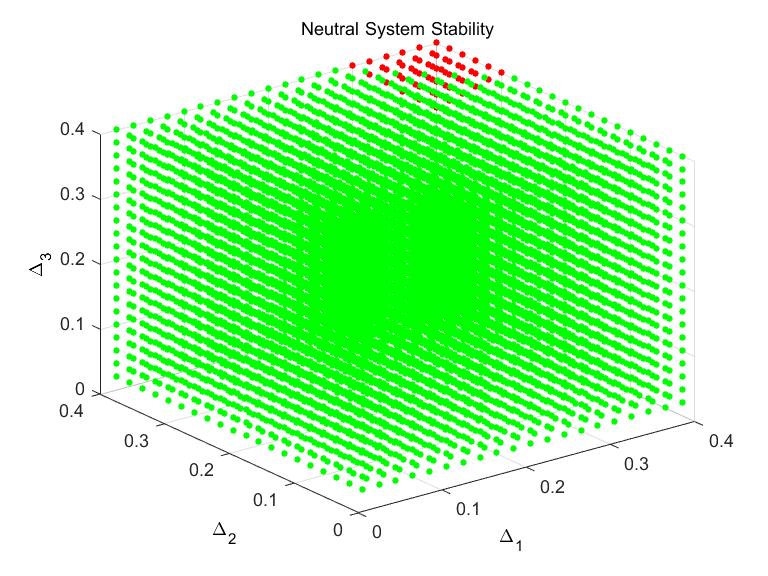}
     \end{subfigure}
     ~\hspace{-.01in}~
     \begin{subfigure}[b]{0.3\textwidth}
         \centering
         \includegraphics[width=\textwidth]{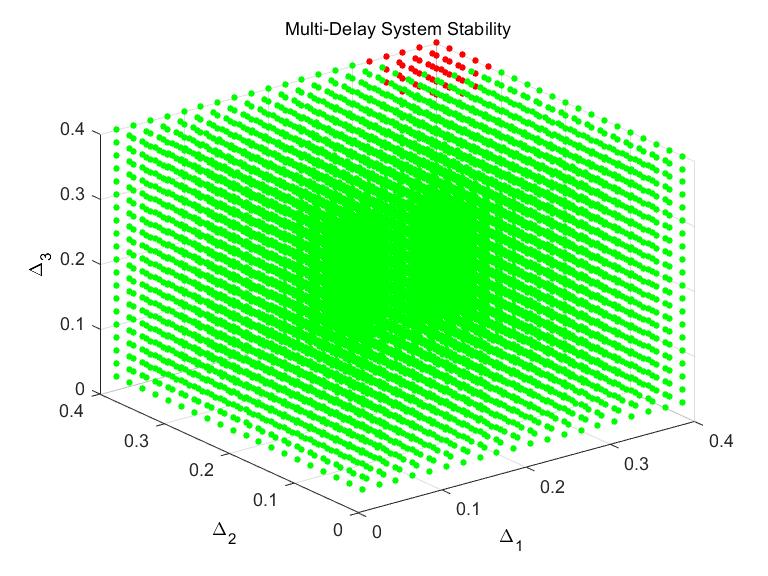}
          \centering
     \end{subfigure}
     ~\hspace{-.01in}~
     \begin{subfigure}[b]{0.3\textwidth}
         \centering
         \includegraphics[width=\textwidth]{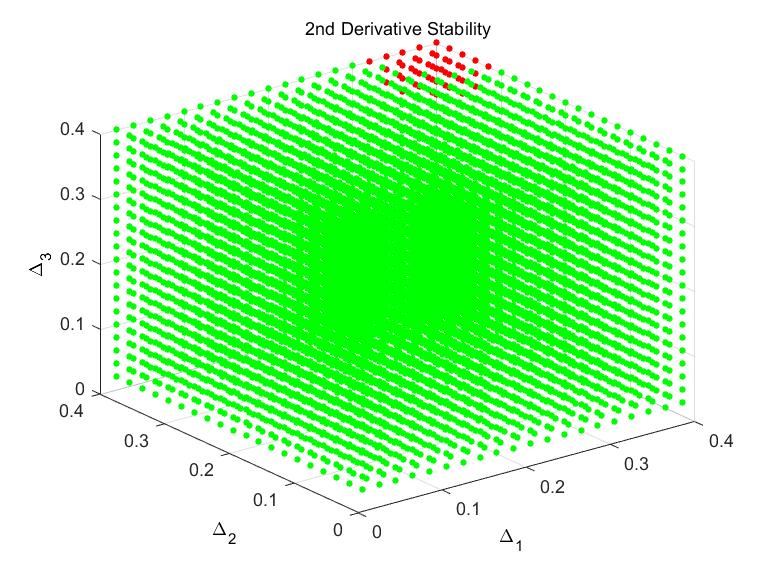}
     \end{subfigure}
\centering
     \begin{subfigure}[b]{0.3\textwidth}
         \centering
         \includegraphics[width=\textwidth]{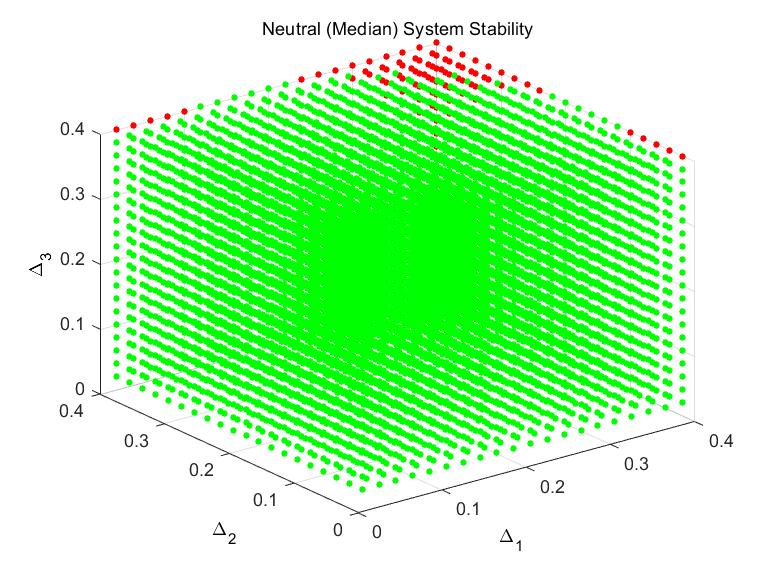}
     \end{subfigure}
     \begin{subfigure}[b]{0.3\textwidth}
         \centering
         \includegraphics[width=\textwidth]{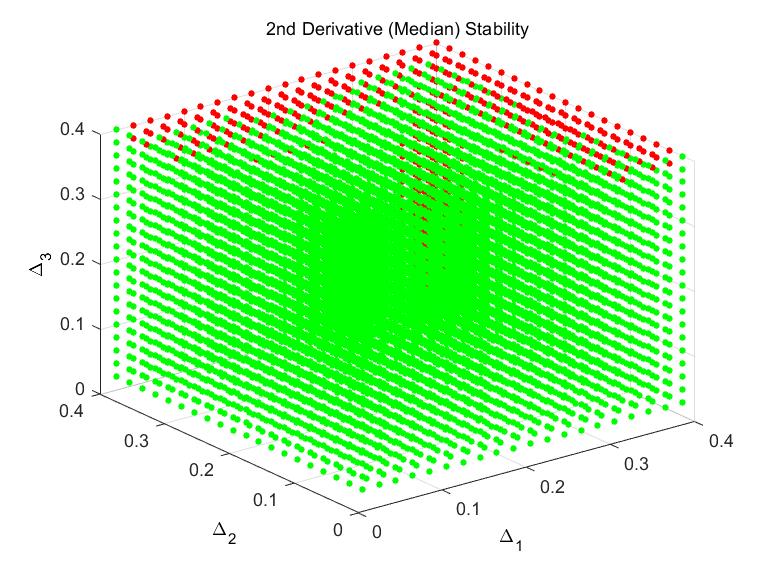}
     \end{subfigure}

\caption{On the first row, we have the neutral approximation with $\Delta^* = \mathbb{E}[\Delta]$ (Left), the three-delay system (Middle), and the second-derivative approximation with $\Delta^* = \mathbb{E}[\Delta]$ (Right). In the second row, we have the neutral approximation with $\Delta^* = \text{median}(\Delta)$ (Left) and the second-derivative approximation with $\Delta^* = \text{median}(\Delta)$. In all of the scatterplots, we have $p_1 = p_2 = p_3 = \frac{1}{3}$, $\alpha_0 = -1$, and $C = -5$. Green points in the scatterplots correspond to points where the system is stable and red points in the scatterplots correspond to points where the system is unstable.}
\label{median_example}
\end{figure}



\section{Explicit Analysis of Two-Delay Model}
\label{section_2_delay}

We examine the two-delay model to get a better understanding of how good our approximations are in this setting. In the two-delay model, we assume that delay $\Delta_1$ occurs with probability $p$ and delay $\Delta_2$ occurs with probability $1-p$, as written below.

\begin{align}
\overset{\bullet}{u}(t) &= \alpha_0 u(t) + \alpha_1 u(t - \Delta_1) + \alpha_2 u(t - \Delta_2)\\
&= \alpha_0 u(t) + C [p \cdot u(t - \Delta_1) + (1-p) \cdot u(t - \Delta_2)]
\end{align}

\noindent In this two-delay setting our neutral approximation is $$\overset{\bullet}{u}(t) = \alpha_0 u(t) + A_0 \cdot u(t - \Delta^*) + A_1 \cdot \overset{\bullet}{u}(t - \Delta^*)$$ and our second-derivative approximation is $$\overset{\bullet}{u}(t) = \alpha_0 u(t) + A_0 \cdot u(t - \Delta^*) + A_1 \cdot \overset{\bullet}{u}(t - \Delta^*) + A_2 \cdot \overset{\bullet \bullet}{u}(t - \Delta^*)$$ with $$A_0 = \alpha_1 + \alpha_2 = C$$
$$A_1 = \alpha_1 (\Delta^* - \Delta_1) + \alpha_2 (\Delta^* - \Delta_2) = C \cdot \mathbb{E}[(\Delta^* - \Delta)]$$ $$A_2 =  \frac{1}{2}[\alpha_1 (\Delta^* - \Delta_1)^2 + \alpha_2 (\Delta^* - \Delta_2)^2] = \frac{C}{2} \cdot \mathbb{E}[(\Delta^* - \Delta)^2].$$

\noindent However, we saw earlier that the choice $\Delta^* = \mathbb{E}[\Delta]$ performed the best in the neutral case, so we will make that choice here as well. Making this choice makes it so that $$A_1 = 0 \hspace{5mm} \text{and} \hspace{5mm} A_2 = \frac{C}{2} \cdot \text{Var}(\Delta). $$ We note that this causes the first derivative term to vanish in our neutral approximation. Recall that before truncating our Taylor expansion, we had $$\overset{\bullet}{u}(t) = \alpha_0 u(t) + C \left[  \sum_{k=0}^{\infty} \frac{1}{k!} \mathbb{E}\left[ (\Delta^* - \Delta)^k  \right] u^{(k)}(t - \Delta^*)  \right]$$ and with the choice of $\Delta^* = \mathbb{E}[\Delta]$, this becomes

\begin{align*}
 \overset{\bullet}{u}(t) &= \alpha_0 u(t) + C \left[  \sum_{k=0}^{\infty} \frac{1}{k!} \mathbb{E}\left[ (\mathbb{E}[\Delta] - \Delta)^k  \right] u^{(k)}(t - \mathbb{E}[\Delta])  \right]\\
&= \alpha_0 u(t) + C \left[  \sum_{k=0}^{\infty} \frac{(-1)^k}{k!} \mathbb{E}\left[ ( \Delta- \mathbb{E}[\Delta] )^k  \right] u^{(k)}(t - \mathbb{E}[\Delta])  \right]\\
&= \alpha_0 u(t) + C \left[  \sum_{k=0}^{\infty} \frac{(-1)^k}{k!} m_k  u^{(k)}(t - \mathbb{E}[\Delta])  \right]
\end{align*}

\noindent where $m_k$ is the $k$-th central moment of $\Delta$. Because of this, we can see that the error terms of our approximations can be expressed in terms of central moments of $\Delta$. We are interested in understanding how approximations would do if we kept higher-order terms in the Taylor expansion. This will lead us to explore some properties of the central moments of $\Delta$ when examining the error of our approximations. We start by stating Lemma \ref{max_variance_lemma} and Lemma \ref{2nd_error_lemma} which explain why our approximations did poorly near certain values of $p$.

\begin{lemma}
\label{lemma41}
The coefficient of the leading-order term in the error for the neutral approximation of the two-delay system with $\Delta^* = \mathbb{E}[\Delta]$ is maximized at $p = \frac{1}{2}$.
\label{max_variance_lemma}
\end{lemma}

\begin{proof}
The leading-order error term for the neutral approximation when $\Delta^* = \mathbb{E}[\Delta]$ is $$\frac{1}{2!} \mathbb{E}[(\Delta^* - \Delta)^2] \overset{\bullet \bullet}{u}(t - \Delta^*).$$ Thus, we want to find the value of $p$ that maximizes the function 

\begin{align*}
f(p) &= \frac{1}{2!} \mathbb{E}[(\Delta^* - \Delta)^2]\\
&= \frac{1}{2!} \left( p ( [p \Delta_1 + (1-p) \Delta_2] - \Delta_1  )^2 + (1-p) ( [p \Delta_1 + (1-p) \Delta_2] - \Delta_2)   \right)\\
&= \frac{1}{2!} p (1-p)(\Delta_1 - \Delta_2)^2
\end{align*}

\noindent Setting the derivative equal to zero, we have that $$f'(p) = \frac{1}{2!} (1 - 2p) (\Delta_1 - \Delta_2)^2 = 0$$ and thus $p = \frac{1}{2}.$ Indeed, $f''(p) < 0$ and $f(p)$ is maximized at this point.

\end{proof}

\noindent Lemma \ref{lemma41} helps to explain why the neutral approximation performs poorly near $p = \frac{1}{2}$, as we saw earlier. 

\begin{lemma}
\label{3rd_moment_lemma}
The magnitude of the coefficient of the leading-order term in the error for the second-derivative approximation of the two-delay system with $\Delta^* = \mathbb{E}[\Delta]$ is maximized at $p = \frac{1}{2} \pm \frac{1}{\sqrt{12}}$.
\label{2nd_error_lemma}
\end{lemma}

\begin{proof}

The leading-order term for the second-derivative approximation when $\Delta^* = \mathbb{E}[\Delta]$ is $$\frac{1}{3!} \mathbb{E}[(\Delta^* - \Delta)^3] \overset{\bullet \bullet}{u}(t - \Delta^*).$$ Thus, we want to find the value of $p$ that maximizes the function 

\begin{align*}
f(p) &= \frac{1}{3!} \mathbb{E}[(\Delta^* - \Delta)^3]\\
&=  \frac{1}{3!} \left( p ( [p \Delta_1 + (1-p) \Delta_2] - \Delta_1)^3 + (1-p) ( [p \Delta_1 + (1-p) \Delta_2] - \Delta_2)^3   \right)\\
&= \frac{1}{3!}(\Delta_1 - \Delta_2)^3 \left[ p (1-p)^3 - (1-p)p^3  \right]\\
\end{align*}

\noindent Setting the derivative equal to zero, we have $$f'(p) = \frac{1}{3!}(\Delta_1 - \Delta_2)^3 (6p^2 - 6p +1) = 0$$ so that $p = \frac{1}{2} \pm \frac{1}{\sqrt{12}}$. One can see that $f''(\frac{1}{2} + \frac{1}{\sqrt{12}}) > 0 $ and $f''(\frac{1}{2} - \frac{1}{\sqrt{12}}) < 0 $ so that we have a minimum and a maximum, respectively. Lastly, $f(\frac{1}{2} + \frac{1}{\sqrt{12}}) = - f(\frac{1}{2} - \frac{1}{\sqrt{12}})$.

\end{proof}

As demonstrated in the above lemmas, it is useful for us to understand what the roots are of the derivatives of the central moments of $\Delta$ in order to understand where the central moments, and thus the error terms in our approximations, get maximized in absolute value. Below we write an expression for the $n^{th}$ central moment of $\Delta$, denoted $m_n$, and its derivative.

\begin{align}
m_n &= \mathbb{E}[(\mathbb{E}[\Delta] - \Delta)^n]\\
&= p ([p \Delta_1 + (1-p) \Delta_2  - \Delta_1])^n + (1-p)([p \Delta_1 + (1-p) \Delta_2] - \Delta_2)^n\\
&= p (1-p)^n (\Delta_2 - \Delta_1)^n + (1-p) p^n (\Delta_1 - \Delta_2)^n\\
&=  p (1-p)  [ (1-p)^{n-1}   + (-1)^n p^{n-1}] (\Delta_2 - \Delta_1)^n
\end{align}

\begin{align}
\frac{d m_n}{dp} &= \frac{d}{dp}\left( \mathbb{E}[(\mathbb{E}[\Delta] - \Delta)^n  \right)\\
&= -(\Delta_2 - \Delta_1)^n (n p + p - 1)(1 - p)^{n-1} - (\Delta_1 - \Delta_2)^n (n (p-1) + p) p^{n-1}\\
&= (\Delta_2 - \Delta_1)^n [ (1 - p) - n p](1 - p)^{n-1} - (\Delta_1 - \Delta_2)^n [ p - n(1 - p)   ] p^{n-1}\\
&= (\Delta_2 - \Delta_1)^{n} \bigg[ [ (1 - p) - n p](1 - p)^{n-1} - (-1)^n [ p - n(1 - p)   ] p^{n-1}  \bigg]
\end{align}

It is easy to compute the roots of $m_n$ when $n$ is small, but it becomes more difficult when $n$ is large. This leads us to approximate the roots of the derivative of the $n^{th}$ central moment.

\subsection{Approximate Extreme Points of the $n^{th}$ Central Moment}

Because the error terms of our approximations contain scaled central moments of $\Delta$, it is useful to get an understanding of what the roots of the derivative of the $n^{th}$ central moment are because this gives us an idea of what values of $p$ maximize the error of our approximations. We can immediately see that if $n$ is even, then $p = \frac{1}{2}$ is a root of $\frac{d}{dp}\left( \mathbb{E}[(\mathbb{E}[\Delta] - \Delta)^n  \right)$. The other roots are more complicated, but we can approximate them when $n$ is large. Consider when $p \neq \frac{1}{2}$ so that either $ p > 1- p$ or $p < 1 - p$. 

In the former case, we have that $(1-p)^{n-1}$ goes to zero faster than $p^{n-1}$ does as $n \to \infty$, so we can approximate one of the roots for large $n$ by examining the roots the dominant term. $$- (\Delta_1 - \Delta_2)^n [ p - n(1 - p)   ] p^{n-1} = 0$$ We see that this equation is satisfied when $$ p - n(1-p) = 0$$ so that $$p = \frac{n}{n+1}.$$ 

\noindent Similarly, if we consider the latter case, the other term is the dominant term for large $n$. $$(\Delta_2 - \Delta_1)^n [ (1 - p) - n p](1 - p)^{n-1} = 0$$ is also satisfied when $$(1 - p) - np = 0$$ so that $$p = \frac{1}{n+1}.$$ Thus, we have that $p = \frac{1}{n+1}$ and $p = \frac{n}{n+1}$ approximate some of the roots of $\frac{d}{dp}\left( \mathbb{E}[(\mathbb{E}[\Delta] - \Delta)^n  \right)$. On the left side of Figure \ref{central_moments_plot} we plot several central moments and we see that they are maximized in magnitude close to $p = \frac{1}{n+1}$ and $p = \frac{n}{n+1}$. However, the leading-order error term of the approximation to the multi-delay system that keeps the first $n-1$ derivatives is $$\frac{(-1)^n}{n!} \mathbb{E}[(\mathbb{E}[\Delta] - \Delta)^n].$$ We plot these leading-order error terms for various values of $n$ on the right side in Figure \ref{central_moments_plot} and we see that the factor of $\frac{1}{n!}$ makes most of these terms quite small so that only examining the error terms up to $n=4$ might suffice in practice. Coincidentally, it turns out that the third and fourth central moments are each maximized in magnitude at $ p = \frac{1}{2} \pm \frac{1}{\sqrt{12}}$. This fact would lead us to believe that the second-derivative approximation would do most poorly around $p = \frac{1}{2} - \frac{1}{\sqrt{12}} \approx .2113$ and $p = \frac{1}{2} + \frac{1}{\sqrt{12}} \approx .7887$. Indeed, the information in Table \ref{Table_Accuracy_2nd} seems to suggest that the second-derivative approximation does the worst somewhat close to these values of $p$. We take a closer look at the accuracy of the second-derivative approximation near one of these values and collect this information in the table in Figure \ref{Table_20to35} to get a better idea of what is going on.

From the information in the table in Figure \ref{Table_20to35}, we see that the second-derivative approximation actually did the worst close to $p = .27$. While this is not precisely the value of $p$ at which the leading-order error terms (namely, those corresponding to $n=3$ and $n=4$) are maximized, it is fairly close by and differs due to the little influence that the higher-ordered error terms have. While it is not completely accurate, determining where the leading-ordered error term is maximized is still a good way of approximating what values of $p$ the approximation will be the least accurate at.

%
%


\begin{figure}
\includegraphics[scale=.43]{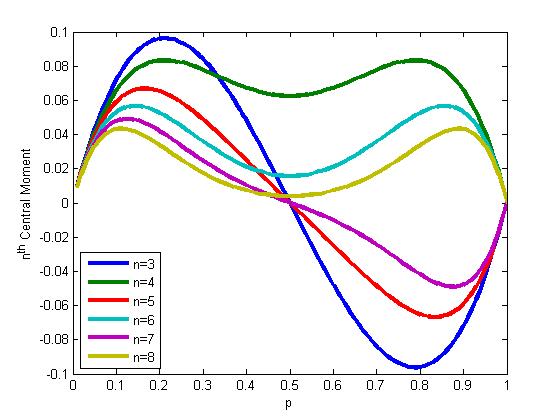}\includegraphics[scale=.43]{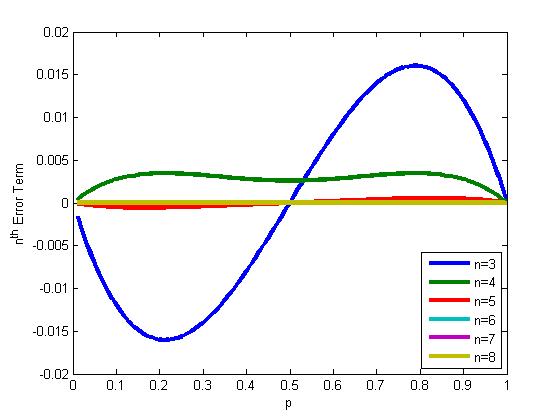}
\caption{Left: Various central moments for $\Delta$ are plotted against $p$ with $\Delta_1 = 1.4$ and $\Delta_2 = .4$. Right: Various error terms in the Taylor expansion are plotted against $p$ with $\Delta_1 = 1.4$ and $\Delta_2 = .4$.}
\label{central_moments_plot}
\end{figure}

\begin{figure}[!hb]
\begin{floatrow}
\ffigbox{%
\hspace{-20mm}
\begin{tabular}{| l | l | l | l | l | l |}
\hline\\ [-2.5ex]
$p$ &  Second Derivative (Mean) \\
\hline
.20   & 96.83\%\\
\hline
.21 &  96.22\% \\
\hline
.22 &  95.68 \%\\
\hline
.23 &   95.21\%\\
\hline
.24 &   94.92\%\\
\hline
.25 &   94.69\%\\
\hline
.26 &   94.60\%\\
\hline
.27 &   94.55\%\\
\hline
.28 &   94.66\%\\
\hline
.29 &   94.90\%\\
\hline
.30 &  95.06\% \\
\hline
.31 &  95.65\% \\
\hline
.32 &  96.24\% \\
\hline
.33 &  96.65\% \\
\hline
.34 &  96.98\% \\
\hline
.35 &  97.31\% \\
\hline
\end{tabular}%
}{
}
\capbtabbox{%
\hspace{-21mm}
  \includegraphics[scale=.55]{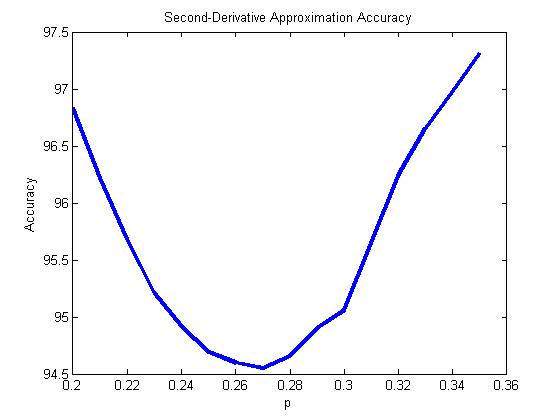}
}{}%
\end{floatrow}%
\caption{The table (left) contains the percentage of points in the second-derivative approximation scatterplots that correctly matched those in the scatterplots corresponding to the multi-delay system with $m = 2$ for various values of $p$ close to $p = \frac{1}{2} - \frac{1}{\sqrt{12}}$, one of the points where the third and fourth central moments are maximized, with $\alpha_0 = -1$ and $C = -5$. The data from this table is plotted here (right) to better visualize how the accuracy of the second-derivative approximation varies with $p$ in the two-delay setting} 
\label{Table_20to35}
\end{figure}

\begin{figure}
\end{figure}

\section{Numerical Examples}
\label{Section5}

\subsection{Comparing Choices of $\Delta^*$ in Neutral Approximation}
\label{subsection_51}

Recall that earlier we approximated a multi-delay DDE with a neutral equation that explicitly depends on only a single constant delay, $\Delta^*$. Keeping in mind that $\Delta^*$ can itself depend on any of the delays present in the multi-delay DDE, earlier we explored how making different choices of $\Delta^*$ caused our neutral equation to approximate where the change in stability occurs in the multi-delay system with varying degrees of accuracy. In this section we provide some visual results from some numerical examples we looked at for the two-delay system that we discussed in Section \ref{subsection31}. More specifically, we provide scatterplots showing where the change in stability occurs in the $\Delta_1$-$\Delta_2$ plane for several choices of $\Delta^*$ and values of $p$. The points in the $\Delta_1$-$\Delta_2$ plane shaded green are points where the system is stable in the sense that the amplitudes of solutions decay as they approach an equilibrium whereas those shaded red are points where the system is unstable in the sense that the amplitudes of solutions increase and approach limit cycles. 

Below, each column of scatterplots corresponds to a specific value of $p$ and each of the first four rows corresponds to a choice of $\Delta^*$ in the neutral equation. In the first row of scatterplots, we consider the choice $\Delta^* = \Delta_1$ which is the delay in the two-delay system that occurs with probability $p$. In the second row we make the choice $\Delta^* = \Delta_2$ which is the delay in the two-delay system that occurs with probability $1 - p$. In the third row, we use a midpoint approximation where we let $\Delta^* = \frac{\Delta_1 + \Delta_2}{2}$. In the fourth row, we use the mean of the discrete random variable $\Delta$ by setting $\Delta^* = \mathbb{E}[\Delta] = p \Delta_1 + (1-p) \Delta_2$. The final row corresponds to the actual multi-delay system (which is only a two-delay system in this case) that we are trying to approximate with our neutral system.


\begin{figure}
     \begin{subfigure}[b]{0.3\textwidth}
       \includegraphics[width=\textwidth]{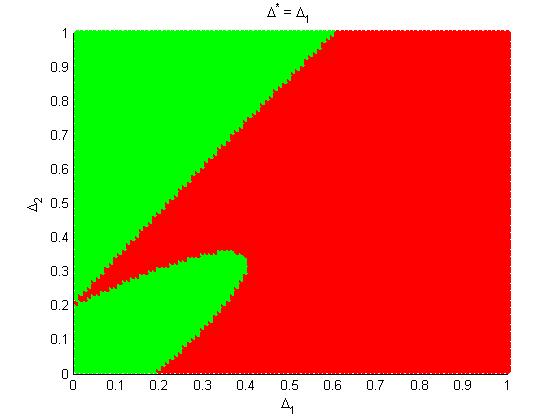}
         \caption{$\Delta^* = \Delta_1, p=.01$}
         \includegraphics[width=\textwidth]{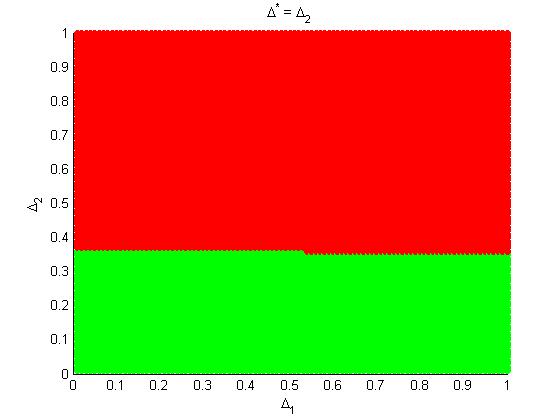}
        \caption{$\Delta^* = \Delta_2, p=.01$}
         \includegraphics[width=\textwidth]{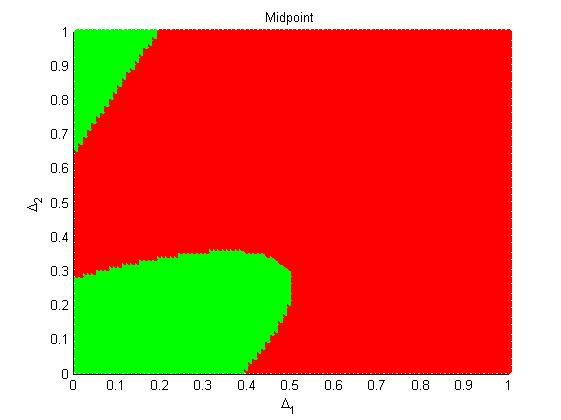}
        \caption{$\Delta^* = \frac{\Delta_1 + \Delta_2}{2}, p=.01$}
         \includegraphics[width=\textwidth]{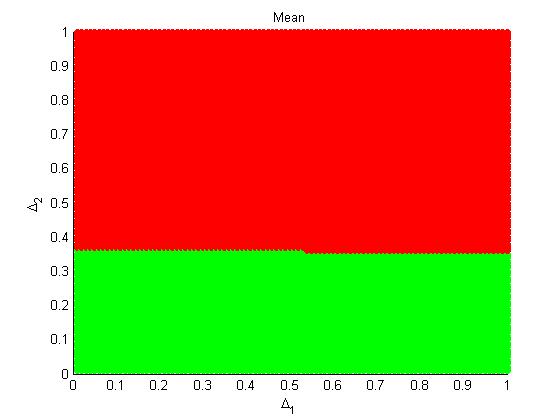}
        \caption{$\Delta^* = \mathbb{E}[\Delta], p=.01$}
         \includegraphics[width=\textwidth]{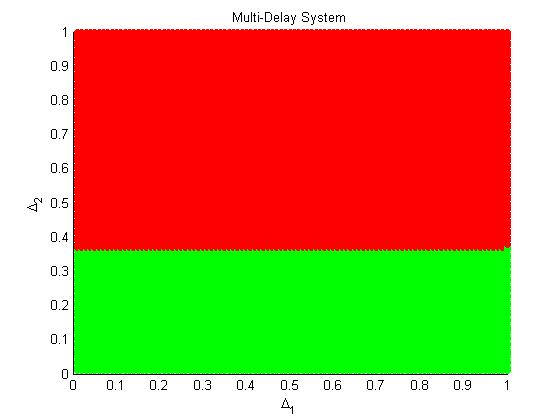}
        \caption{Multi-Delay, $p = .01$}
     \end{subfigure}
\hspace*{\fill}
     \begin{subfigure}[b]{0.3\textwidth}
       \includegraphics[width=\textwidth]{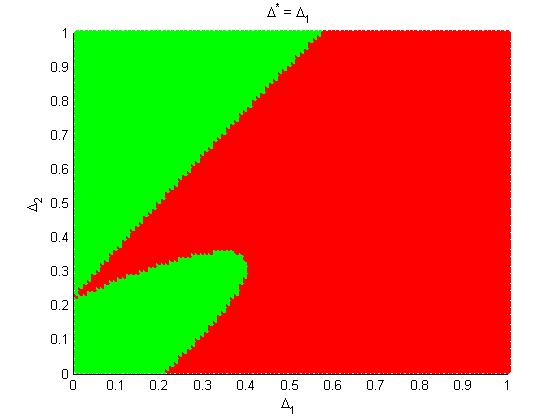}
         \caption{$\Delta^* = \Delta_1, p=.1$}
         \includegraphics[width=\textwidth]{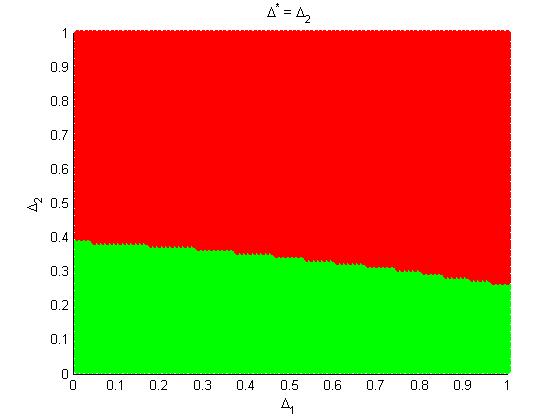}
        \caption{$\Delta^* = \Delta_2, p=.1$}
         \includegraphics[width=\textwidth]{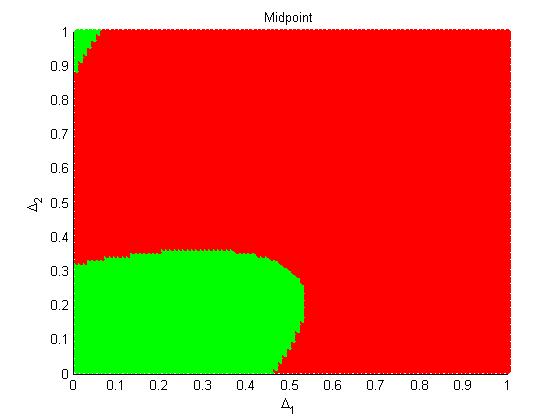}
        \caption{$\Delta^* = \frac{\Delta_1 + \Delta_2}{2}, p=.1$}
         \includegraphics[width=\textwidth]{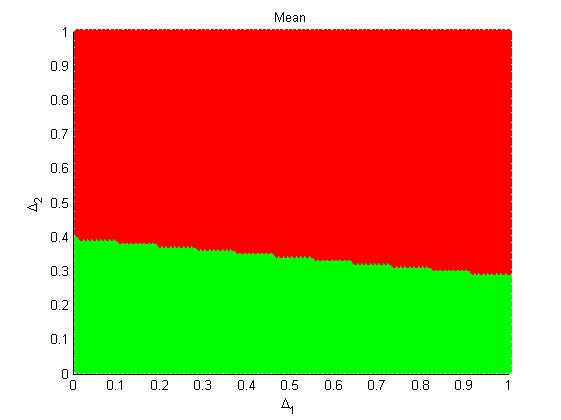}
        \caption{$\Delta^* = \mathbb{E}[\Delta], p=.1$}
         \includegraphics[width=\textwidth]{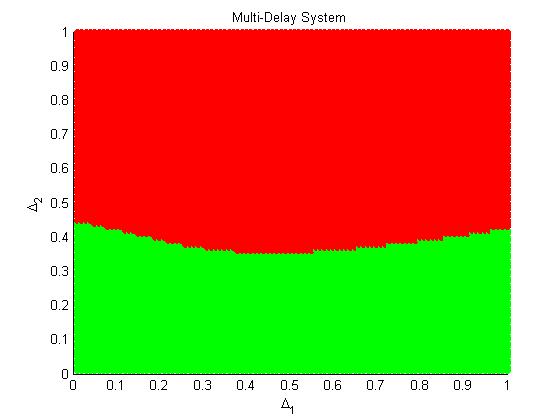}
        \caption{Multi-Delay, $p = .1$}
     \end{subfigure}
\hspace*{\fill}
     \begin{subfigure}[b]{0.3\textwidth}
       \includegraphics[width=\textwidth]{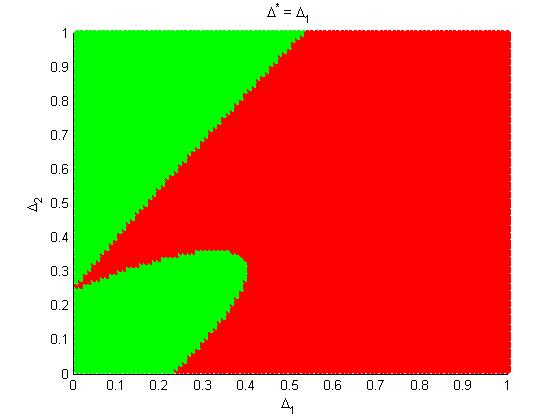}
         \caption{$\Delta^* = \Delta_1, p=.2$}
         \includegraphics[width=\textwidth]{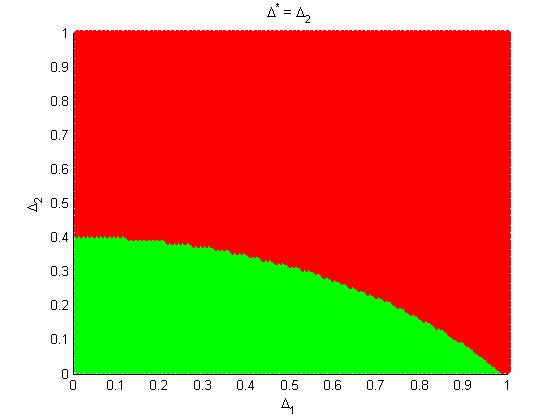}
        \caption{$\Delta^* = \Delta_2, p=.2$}
         \includegraphics[width=\textwidth]{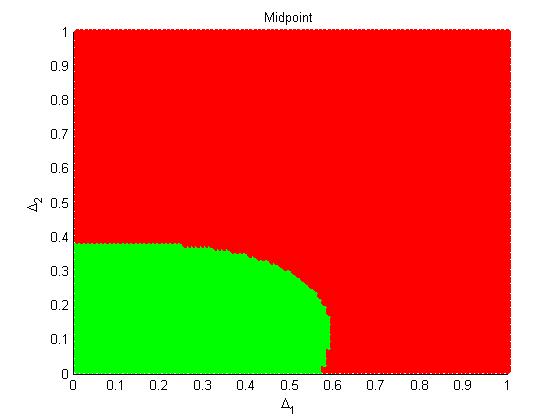}
        \caption{$\Delta^* = \frac{\Delta_1 + \Delta_2}{2}, p=.2$}
         \includegraphics[width=\textwidth]{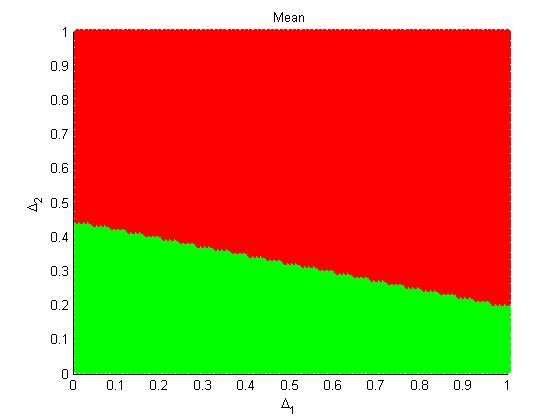}
        \caption{$\Delta^* = \mathbb{E}[\Delta], p=.2$}
         \includegraphics[width=\textwidth]{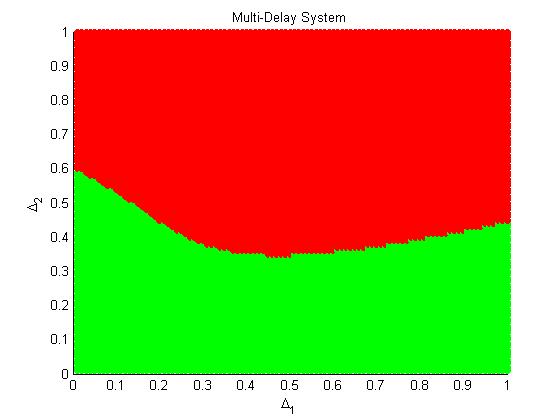}
        \caption{Multi-Delay, $p = .2$}
     \end{subfigure}
\end{figure}


\begin{figure}
     \begin{subfigure}[b]{0.3\textwidth}
       \includegraphics[width=\textwidth]{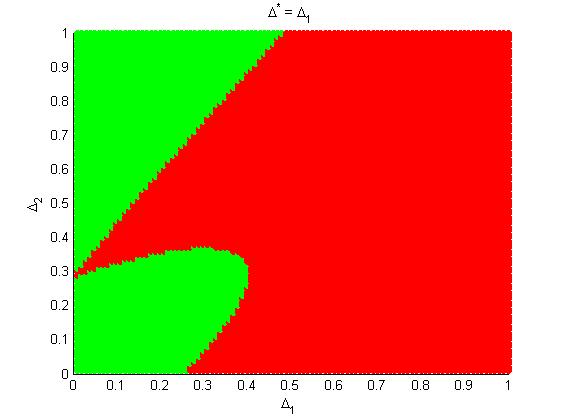}
         \caption{$\Delta^* = \Delta_1, p=.3$}
         \includegraphics[width=\textwidth]{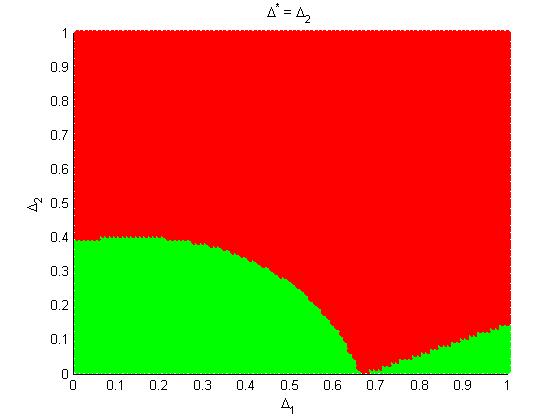}
        \caption{$\Delta^* = \Delta_2, p=.3$}
         \includegraphics[width=\textwidth]{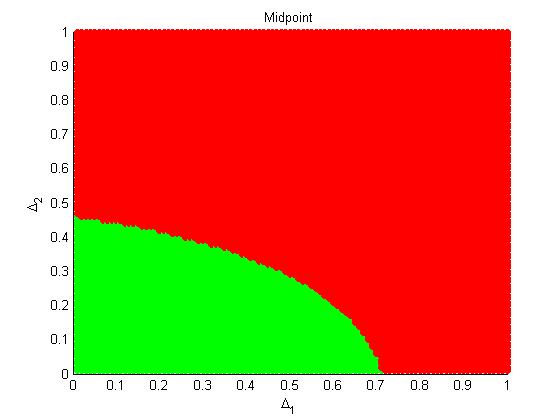}
        \caption{$\Delta^* = \frac{\Delta_1 + \Delta_2}{2}, p=.3$}
         \includegraphics[width=\textwidth]{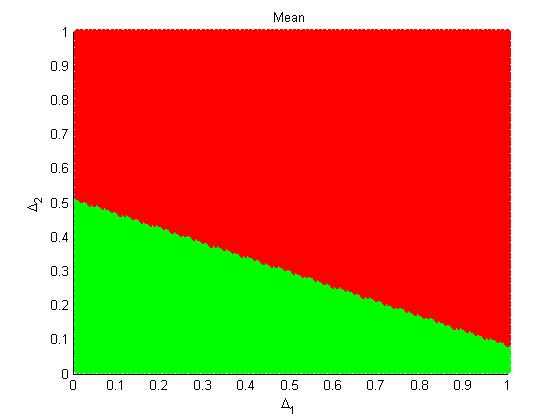}
        \caption{$\Delta^* = \mathbb{E}[\Delta], p=.3$}
         \includegraphics[width=\textwidth]{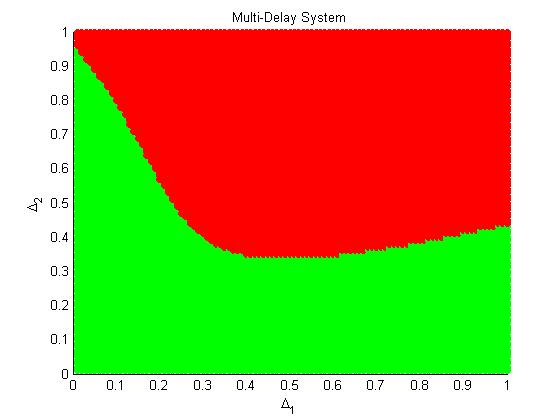}
        \caption{Multi-Delay, $p = .3$}
     \end{subfigure}
\hspace*{\fill}
     \begin{subfigure}[b]{0.3\textwidth}
       \includegraphics[width=\textwidth]{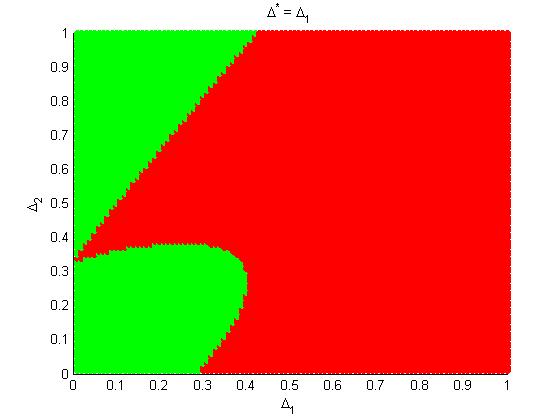}
         \caption{$\Delta^* = \Delta_1, p=.4$}
         \includegraphics[width=\textwidth]{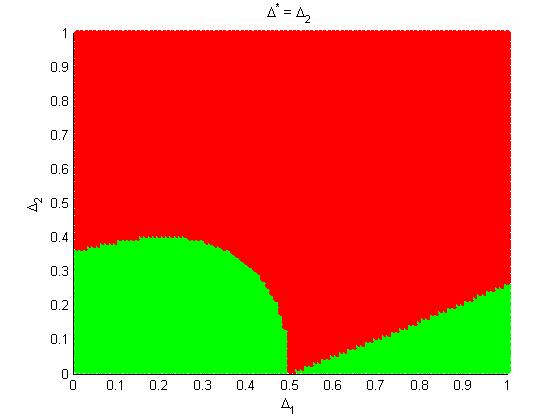}
        \caption{$\Delta^* = \Delta_2, p=.4$}
         \includegraphics[width=\textwidth]{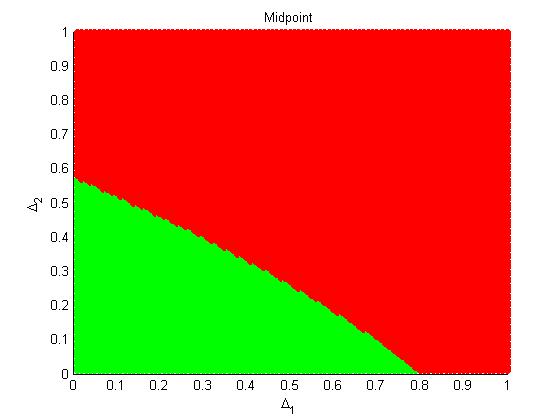}
        \caption{$\Delta^* = \frac{\Delta_1 + \Delta_2}{2}, p=.4$}
         \includegraphics[width=\textwidth]{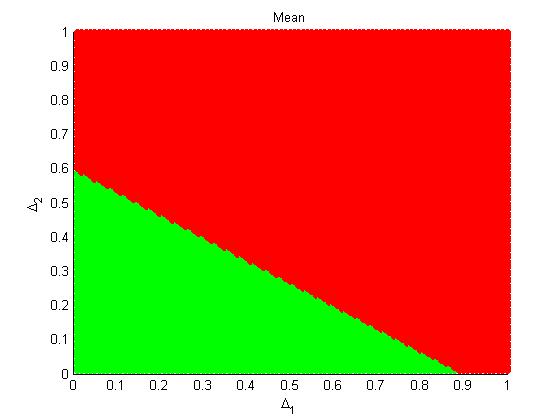}
        \caption{$\Delta^* = \mathbb{E}[\Delta], p=.4$}
         \includegraphics[width=\textwidth]{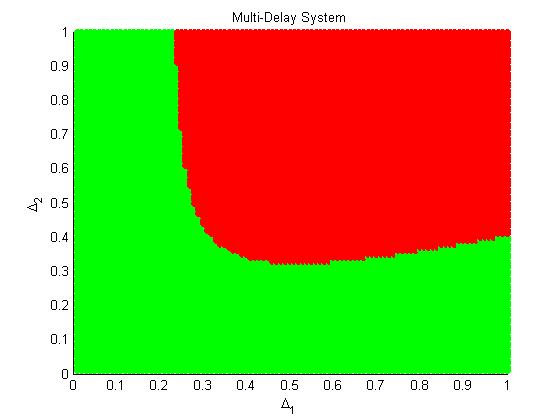}
        \caption{Multi-Delay, $p = .4$}
     \end{subfigure}
\hspace*{\fill}
     \begin{subfigure}[b]{0.3\textwidth}
       \includegraphics[width=\textwidth]{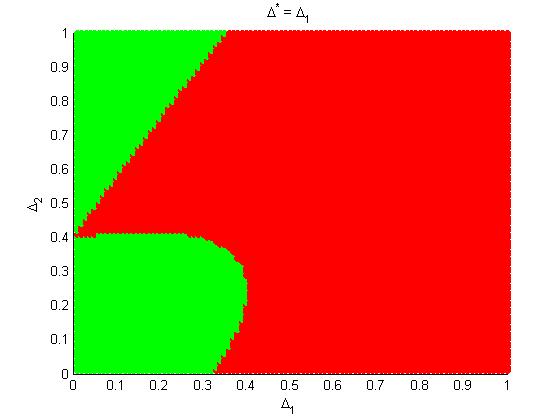}
         \caption{$\Delta^* = \Delta_1, p=.5$}
         \includegraphics[width=\textwidth]{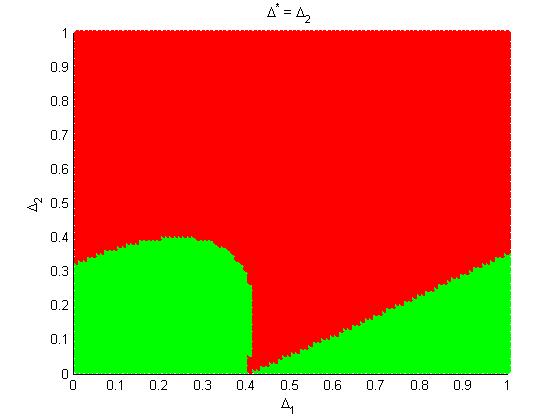}
        \caption{$\Delta^* = \Delta_2, p=.5$}
         \includegraphics[width=\textwidth]{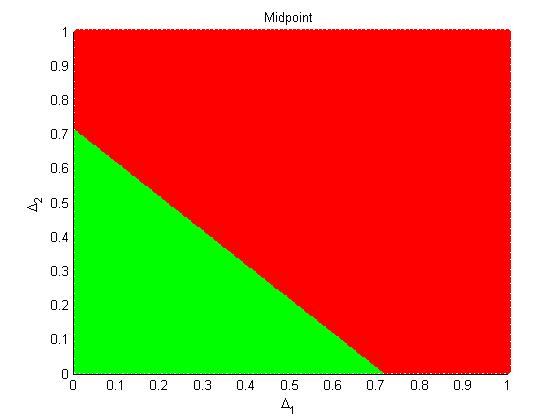}
        \caption{$\Delta^* = \frac{\Delta_1 + \Delta_2}{2}, p=.5$}
         \includegraphics[width=\textwidth]{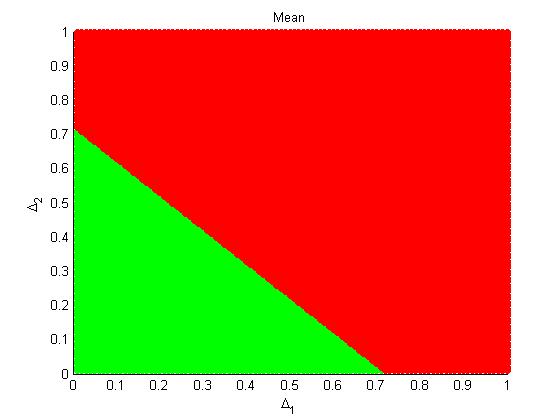}
        \caption{$\Delta^* = \mathbb{E}[\Delta], p=.5$}
         \includegraphics[width=\textwidth]{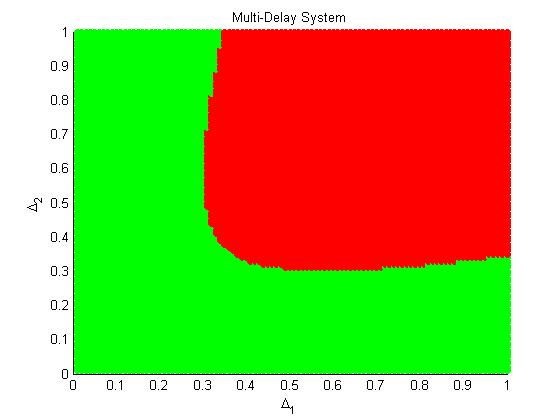}
        \caption{Multi-Delay, $p = .5$}
     \end{subfigure}
\end{figure}


\begin{figure}
     \begin{subfigure}[b]{0.3\textwidth}
       \includegraphics[width=\textwidth]{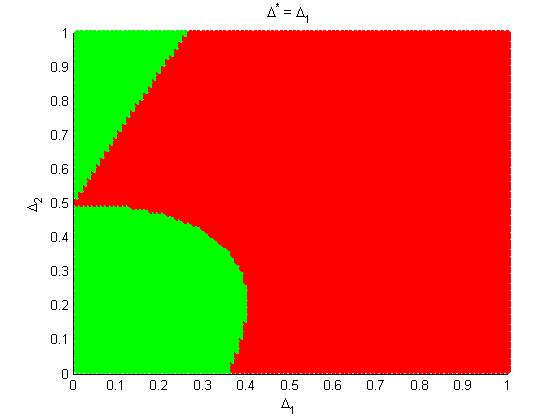}
         \caption{$\Delta^* = \Delta_1, p=.6$}
         \includegraphics[width=\textwidth]{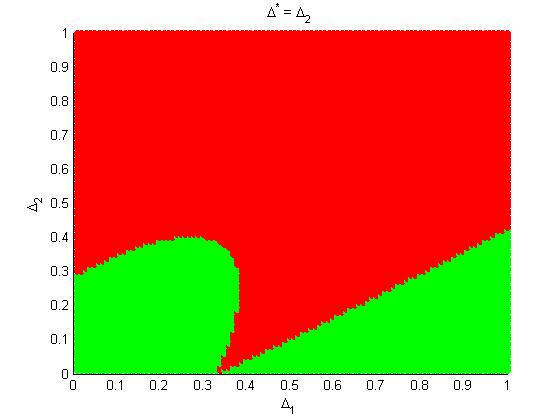}
        \caption{$\Delta^* = \Delta_2, p=.6$}
         \includegraphics[width=\textwidth]{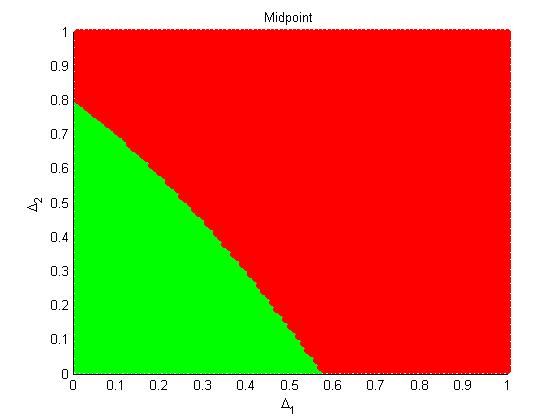}
        \caption{$\Delta^* = \frac{\Delta_1 + \Delta_2}{2}, p=.6$}
         \includegraphics[width=\textwidth]{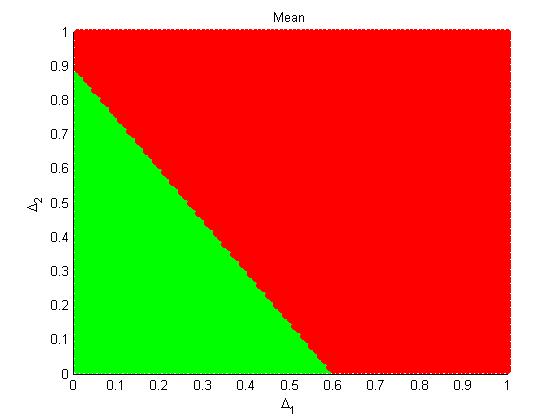}
        \caption{$\Delta^* = \mathbb{E}[\Delta], p=.6$}
         \includegraphics[width=\textwidth]{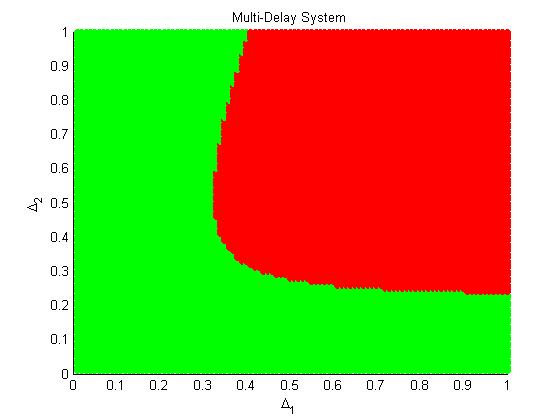}
        \caption{Multi-Delay, $p = .6$}
     \end{subfigure}
\hspace*{\fill}
     \begin{subfigure}[b]{0.3\textwidth}
       \includegraphics[width=\textwidth]{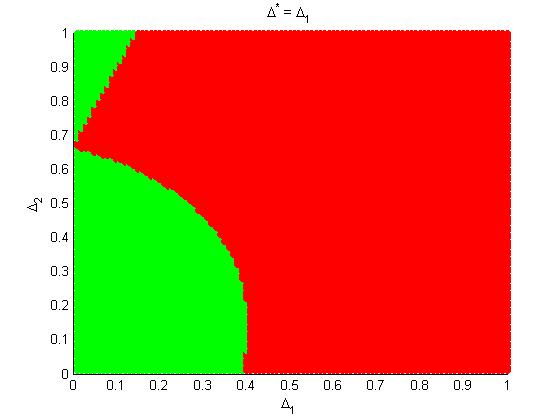}
         \caption{$\Delta^* = \Delta_1, p=.7$}
         \includegraphics[width=\textwidth]{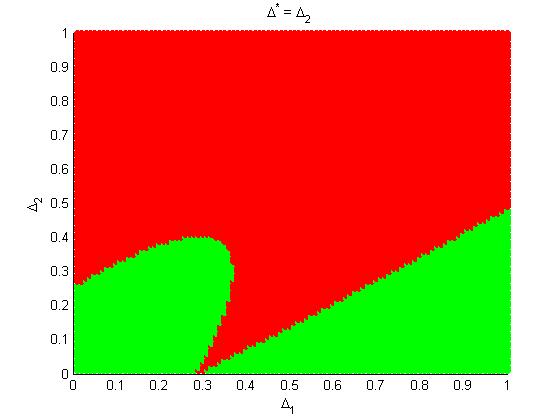}
        \caption{$\Delta^* = \Delta_2, p=.7$}
         \includegraphics[width=\textwidth]{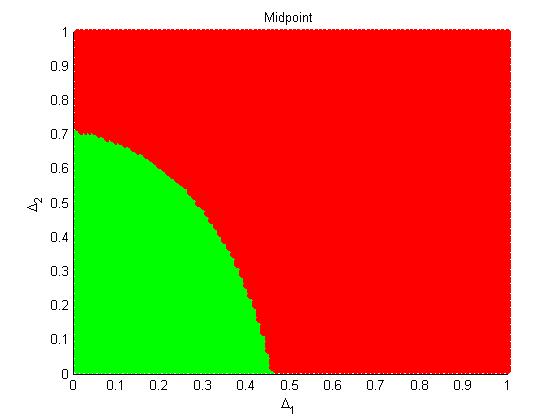}
        \caption{$\Delta^* = \frac{\Delta_1 + \Delta_2}{2}, p=.7$}
         \includegraphics[width=\textwidth]{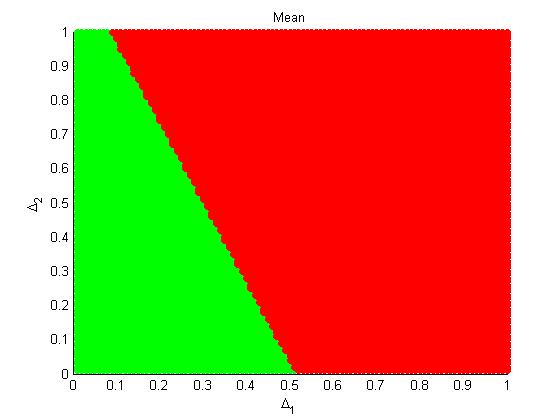}
        \caption{$\Delta^* = \mathbb{E}[\Delta], p=.7$}
         \includegraphics[width=\textwidth]{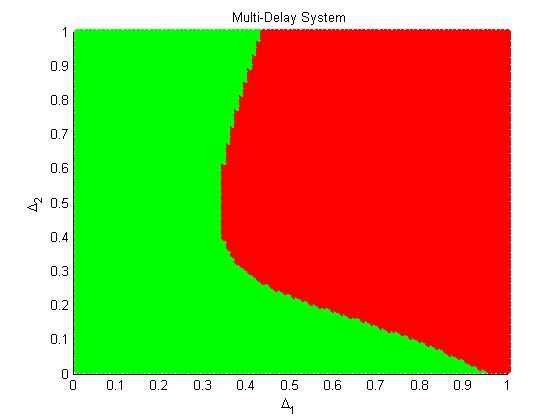}
        \caption{Multi-Delay, $p = .7$}
     \end{subfigure}
\hspace*{\fill}
     \begin{subfigure}[b]{0.3\textwidth}
       \includegraphics[width=\textwidth]{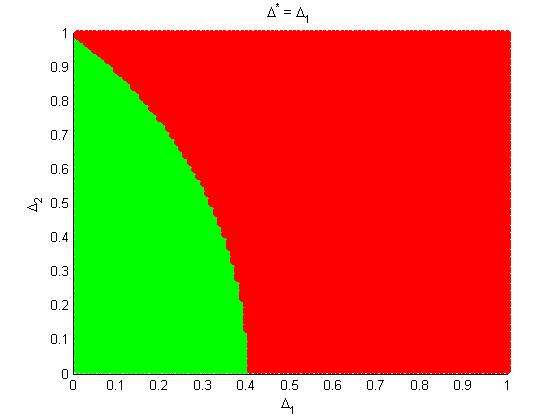}
         \caption{$\Delta^* = \Delta_1, p=.8$}
         \includegraphics[width=\textwidth]{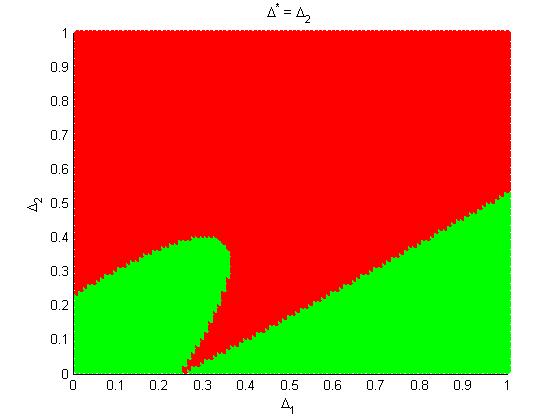}
        \caption{$\Delta^* = \Delta_2, p=.8$}
         \includegraphics[width=\textwidth]{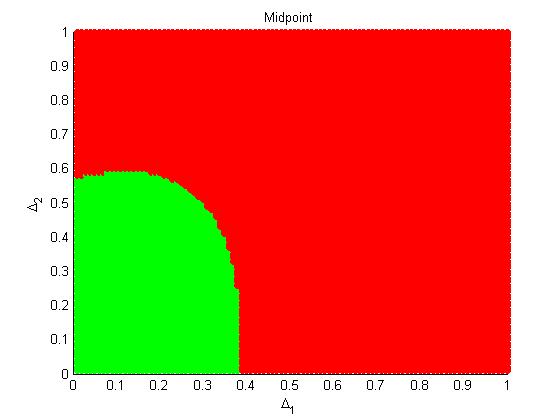}
        \caption{$\Delta^* = \frac{\Delta_1 + \Delta_2}{2}, p=.8$}
         \includegraphics[width=\textwidth]{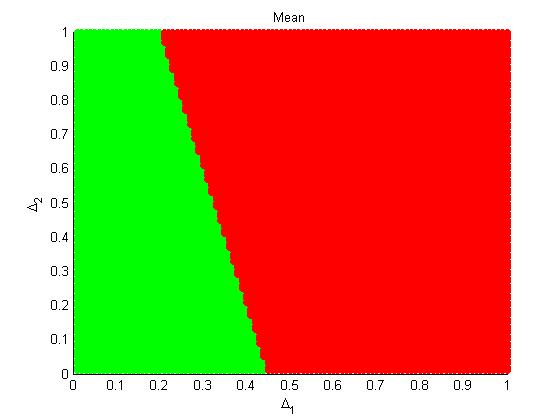}
        \caption{$\Delta^* = \mathbb{E}[\Delta], p=.8$}
         \includegraphics[width=\textwidth]{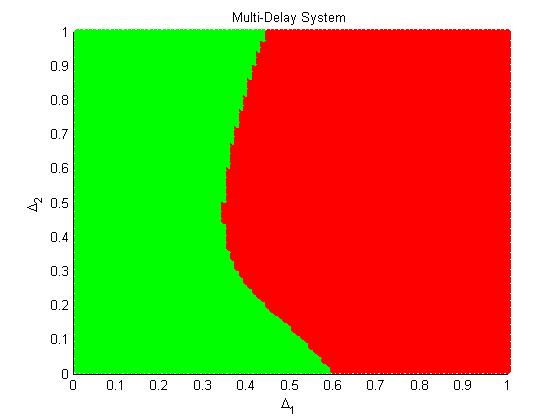}
        \caption{Multi-Delay, $p = .8$}
     \end{subfigure}
\end{figure}


\begin{figure}
\hspace{0mm}
     \begin{subfigure}[b]{0.3\textwidth}
       \includegraphics[width=\textwidth]{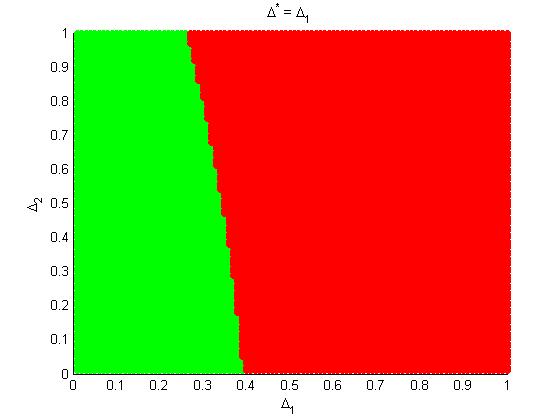}
         \caption{$\Delta^* = \Delta_1, p=.9$}
         \includegraphics[width=\textwidth]{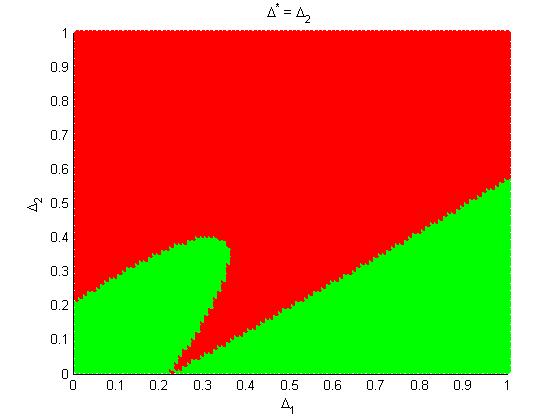}
        \caption{$\Delta^* = \Delta_2, p=.9$}
         \includegraphics[width=\textwidth]{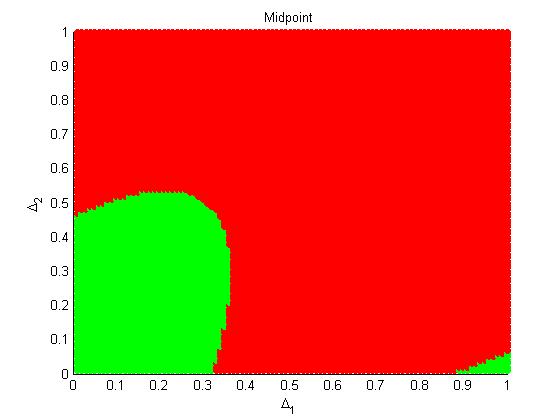}
        \caption{$\Delta^* = \frac{\Delta_1 + \Delta_2}{2}, p=.9$}
         \includegraphics[width=\textwidth]{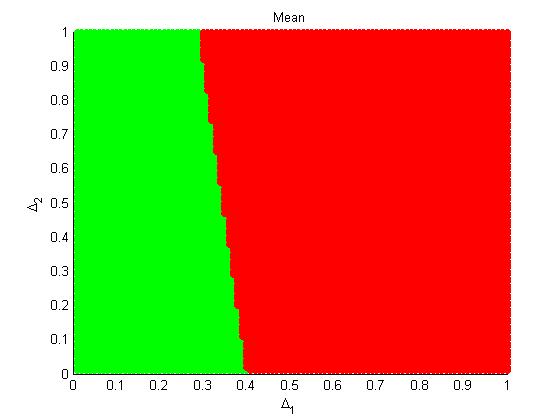}
        \caption{$\Delta^* = \mathbb{E}[\Delta], p=.9$}
         \includegraphics[width=\textwidth]{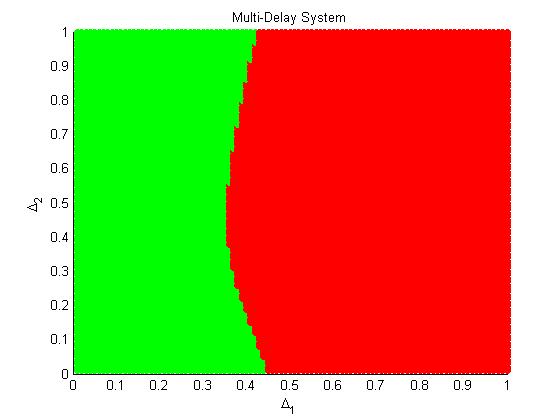}
        \caption{Multi-Delay, $p = .9$}
     \end{subfigure}
\hspace{10mm}
     \begin{subfigure}[b]{0.3\textwidth}
       \includegraphics[width=\textwidth]{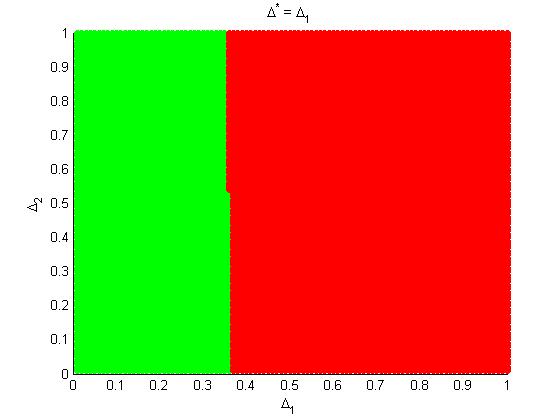}
         \caption{$\Delta^* = \Delta_1, p=.99$}
         \includegraphics[width=\textwidth]{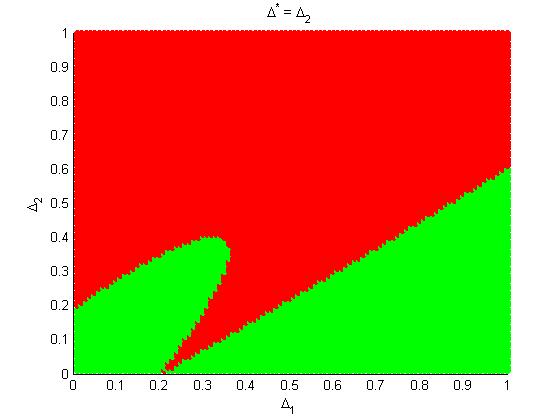}
        \caption{$\Delta^* = \Delta_2, p=.99$}
         \includegraphics[width=\textwidth]{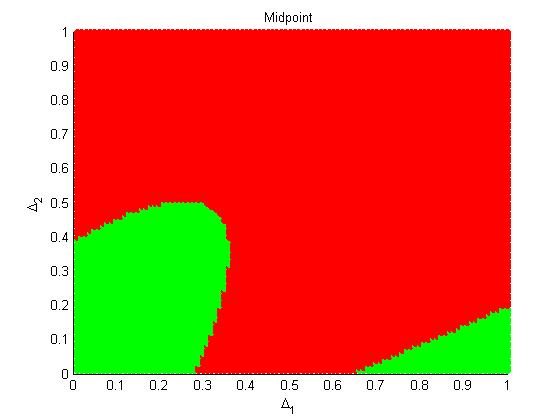}
        \caption{$\Delta^* = \frac{\Delta_1 + \Delta_2}{2}, p=.99$}
         \includegraphics[width=\textwidth]{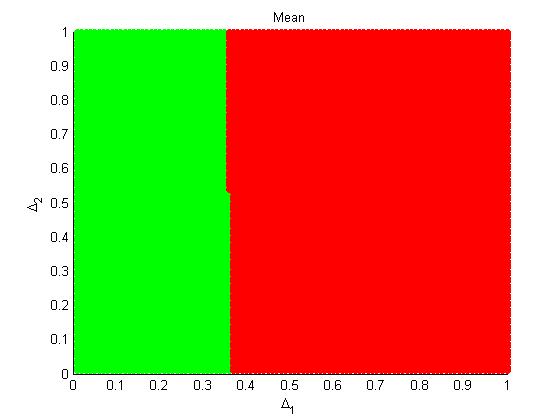}
        \caption{$\Delta^* = \mathbb{E}[\Delta], p=.99$}
         \includegraphics[width=\textwidth]{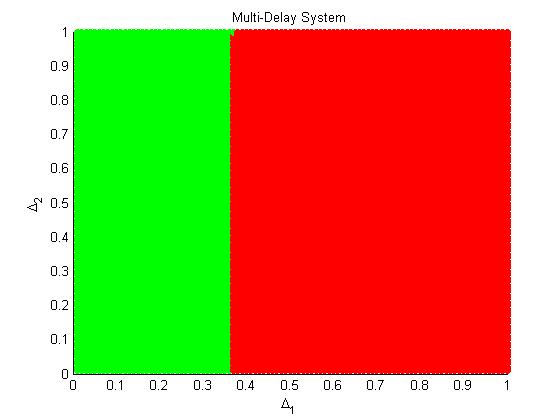}
        \caption{Multi-Delay, $p = .99$}
     \end{subfigure}
\end{figure}


\subsection{Comparing Neutral and Second-Derivative Approximtions}
\label{subsection_52}

Earlier we saw that the various choices of $\Delta^*$ in the neutral approximation affected how accurately the neutral system approximated the multi-delay system. In addition, we saw that choosing $\Delta^* = \mathbb{E}[\Delta]$ gave the neutral approximation the best accuracy in approximating the multi-delay system, which we explained by noting that this choice of $\Delta^*$ minimizes the magnitude of the leading-order error term of the Taylor expansion (essentially just using the fact that the mean minimizes the $L^2$-norm). The good performance of the mean in the neutral setting led us to use this same choice of $\Delta^*$ in the second-derivative approximation. We also saw that the neutral approximation with $\Delta^* = \mathbb{E}[\Delta]$ of the two-delay system does poorly near $p = \frac{1}{2}$ due to the leading-order error term in the Taylor expansion being maximized for this value of $p$. We saw numerically that the accuracy of the neutral approximation would get as low as about $72 \%$ in the two-delay setting. This motivated us to keep the second derivative term to improve upon the neutral approximation as we discussed in Section \ref{subsection_32}. We saw that this led to a significant improvement in the approximation as the second-derivative approximation's accuracy was never worse than $94 \%$ across all of the values of $p$ that we tested. 

In this section we supply more visual aids to seeing how well the neutral and second-derivative approximations estimate where in the $\Delta_1$-$\Delta_2$ plane the change in stability occurs in the two-delay system by showing more scatterplots. Below we have several rows of scatterplots, where each row corresponds to a specific value of $p$ and contains three scatterplots. The first scatterplot in each row compares the stability of the neutral approximation to the stability of the multi-delay system, the second scatterplot in each row compares the stability of the second-derivative approximation to the stability of the multi-delay system, and the third scatterplot shows the stability of the multi-delay system alone. For the first two scatterplots in each row, the points in the $\Delta_1$-$\Delta_2$ plane shaded green represent points where both the approximation and the multi-delay system are stable, points shaded red correspond to points where both the approximation and the multi-delay system are unstable, points shaded yellow represent points where the approximation is unstable, but the multi-delay system is stable, and points shaded blue correspond to points where the multi-delay system is unstable, but the approximation being used is stable. From this, we can observe that the region of instability for the multi-delay system appears to be a subset of the region of instability for the neutral approximation. On the other hand, the region of instability of the second-derivative approximation appears to almost be a subset of the region of instability of the multi-delay system, but we can see that it is not quite a true subset for most values of $p$.

\begin{figure}
     \begin{subfigure}[b]{0.3\textwidth}
         \centering
       \includegraphics[width=\textwidth]{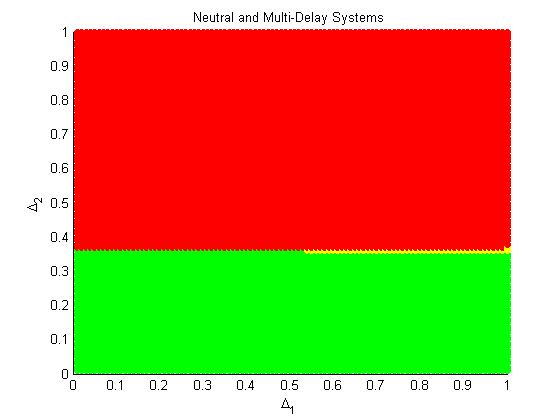}
         \caption{Neutral, $p = .01$}
     \end{subfigure}
     ~\hspace{-.01in}~
     \begin{subfigure}[b]{0.3\textwidth}
         \centering
         \includegraphics[width=\textwidth]{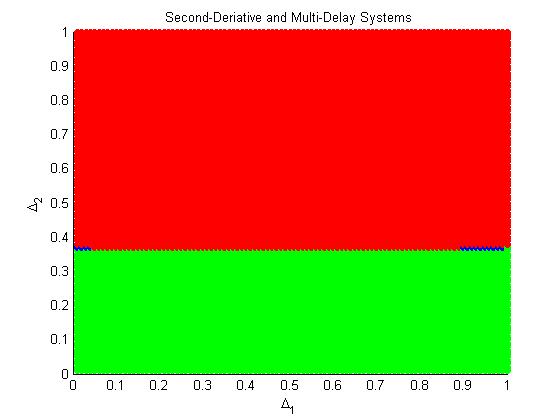}
          \centering
        \caption{$2^{nd}$-Derivative, $p = .01$}
     \end{subfigure}
     ~\hspace{-.01in}~
     \begin{subfigure}[b]{0.3\textwidth}
         \centering
         \includegraphics[width=\textwidth]{./Figures/multidelay/p01/multi.jpg}
         \caption{Multi-Delay, $p = .01$}
     \end{subfigure}
\end{figure}

\begin{figure}
      \centering
     \begin{subfigure}[b]{0.3\textwidth}
         \centering
       \includegraphics[width=\textwidth]{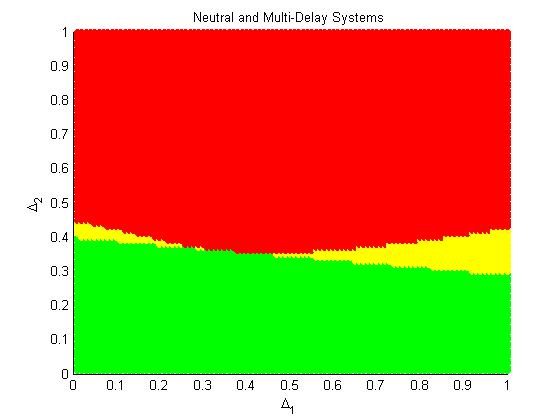}
         \caption{Neutral, $p = .1$}
     \end{subfigure}
     ~\hspace{-.01in}~
     \begin{subfigure}[b]{0.3\textwidth}
         \centering
         \includegraphics[width=\textwidth]{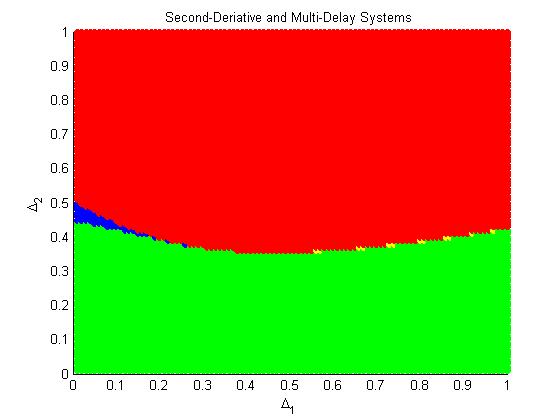}
          \centering
        \caption{$2^{nd}$-Derivative, $p = .1$}
     \end{subfigure}
     ~\hspace{-.01in}~
     \begin{subfigure}[b]{0.3\textwidth}
         \centering
         \includegraphics[width=\textwidth]{./Figures/multidelay/p10/multi.jpg}
         \caption{Multi-Delay, $p = .1$}
     \end{subfigure}

     \begin{subfigure}[b]{0.3\textwidth}
         \centering
       \includegraphics[width=\textwidth]{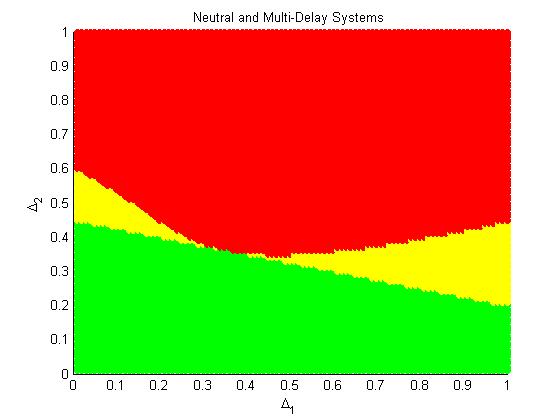}
         \caption{Neutral, $p = .2$}
     \end{subfigure}
     ~\hspace{-.01in}~
     \begin{subfigure}[b]{0.3\textwidth}
         \centering
         \includegraphics[width=\textwidth]{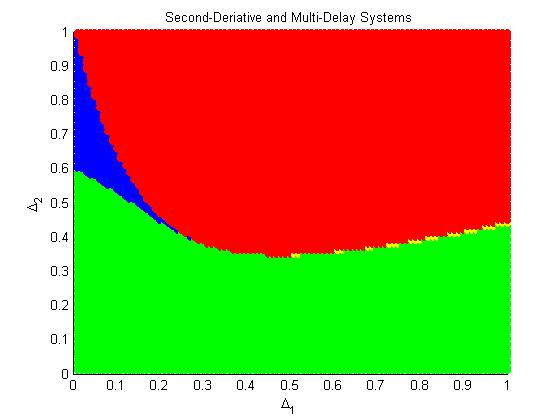}
          \centering
        \caption{$2^{nd}$-Derivative, $p = .2$}
     \end{subfigure}
     ~\hspace{-.01in}~
     \begin{subfigure}[b]{0.3\textwidth}
         \centering
         \includegraphics[width=\textwidth]{./Figures/multidelay/p20/multi.jpg}
         \caption{Multi-Delay, $p = .2$}
     \end{subfigure}
     ~\hspace{-.01in}~
\begin{subfigure}[b]{0.3\textwidth}
         \centering
       \includegraphics[width=\textwidth]{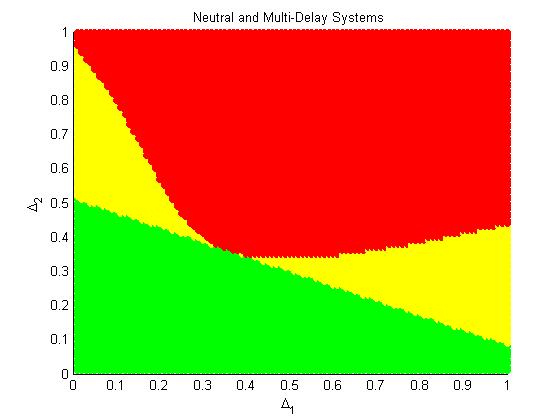}
         \caption{Neutral, $p = .3$}
     \end{subfigure}
     ~\hspace{-.01in}~
     \begin{subfigure}[b]{0.3\textwidth}
         \centering
         \includegraphics[width=\textwidth]{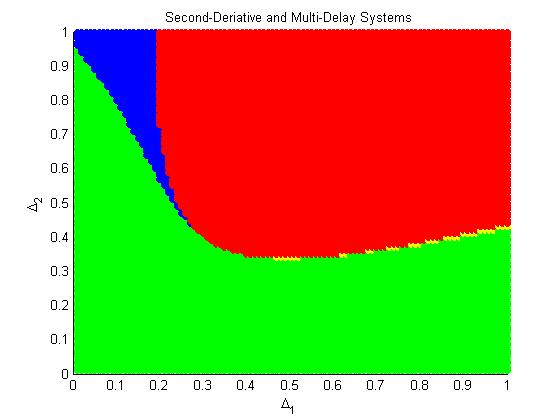}
          \centering
        \caption{$2^{nd}$-Derivative, $p = .3$}
     \end{subfigure}
     ~\hspace{-.01in}~
     \begin{subfigure}[b]{0.3\textwidth}
         \centering
         \includegraphics[width=\textwidth]{./Figures/multidelay/p30/multi.jpg}
         \caption{Multi-Delay, $p = .3$}
     \end{subfigure}
     ~\hspace{-.01in}~
\begin{subfigure}[b]{0.3\textwidth}
         \centering
       \includegraphics[width=\textwidth]{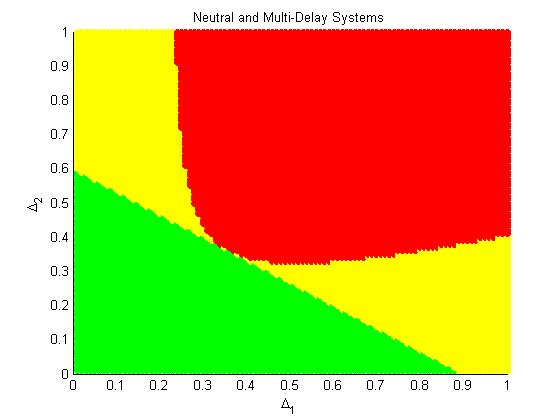}
         \caption{Neutral, $p = .4$}
     \end{subfigure}
     ~\hspace{-.01in}~
     \begin{subfigure}[b]{0.3\textwidth}
         \centering
         \includegraphics[width=\textwidth]{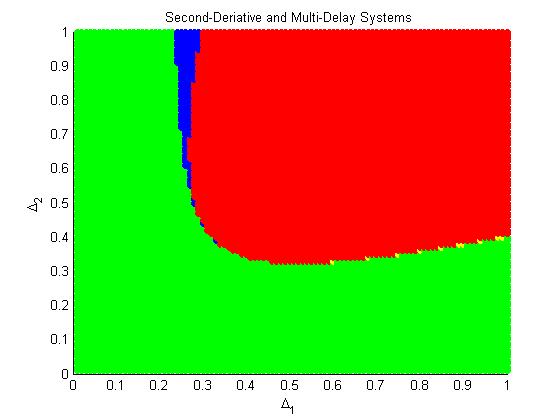}
          \centering
        \caption{$2^{nd}$-Derivative, $p = .4$}
     \end{subfigure}
     ~\hspace{-.01in}~
     \begin{subfigure}[b]{0.3\textwidth}
         \centering
         \includegraphics[width=\textwidth]{./Figures/multidelay/p40/multi.jpg}
         \caption{Multi-Delay, $p = .4$}
     \end{subfigure}
     \begin{subfigure}[b]{0.3\textwidth}
         \centering
       \includegraphics[width=\textwidth]{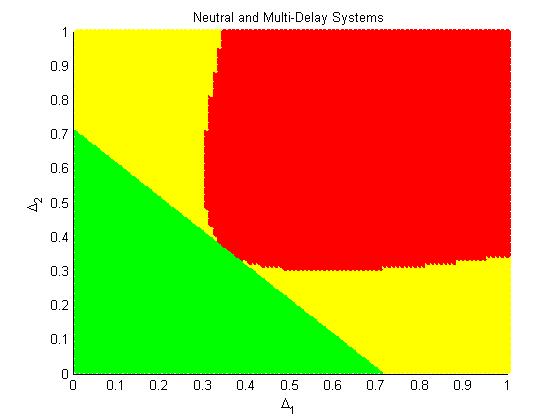}
         \caption{Neutral, $p = .5$}
     \end{subfigure}
     ~\hspace{-.01in}~
     \begin{subfigure}[b]{0.3\textwidth}
         \centering
         \includegraphics[width=\textwidth]{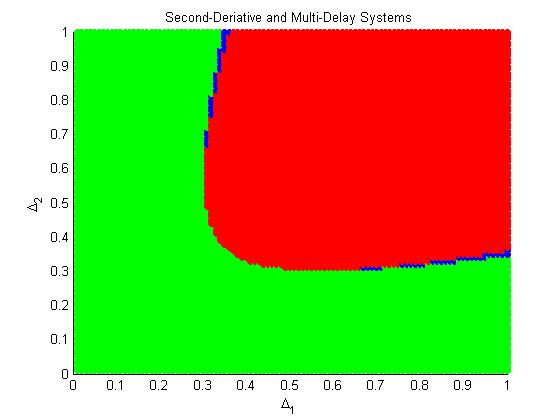}
          \centering
        \caption{$2^{nd}$-Derivative, $p = .5$}
     \end{subfigure}
     ~\hspace{-.01in}~
     \begin{subfigure}[b]{0.3\textwidth}
         \centering
         \includegraphics[width=\textwidth]{./Figures/multidelay/p50/multi.jpg}
         \caption{Multi-Delay, $p = .5$}
     \end{subfigure}
\end{figure}


\begin{figure}
      \centering
     \begin{subfigure}[b]{0.3\textwidth}
         \centering
       \includegraphics[width=\textwidth]{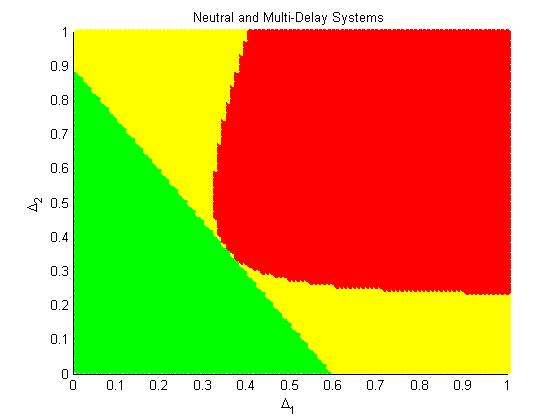}
         \caption{Neutral, $p = .6$}
     \end{subfigure}
     ~\hspace{-.01in}~
     \begin{subfigure}[b]{0.3\textwidth}
         \centering
         \includegraphics[width=\textwidth]{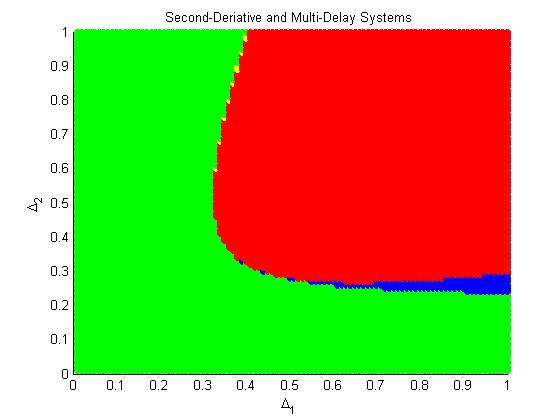}
          \centering
        \caption{$2^{nd}$-Derivative, $p = .6$}
     \end{subfigure}
     ~\hspace{-.01in}~
     \begin{subfigure}[b]{0.3\textwidth}
         \centering
         \includegraphics[width=\textwidth]{./Figures/multidelay/p60/multi.jpg}
         \caption{Multi-Delay, $p = .6$}
     \end{subfigure}

     \begin{subfigure}[b]{0.3\textwidth}
         \centering
       \includegraphics[width=\textwidth]{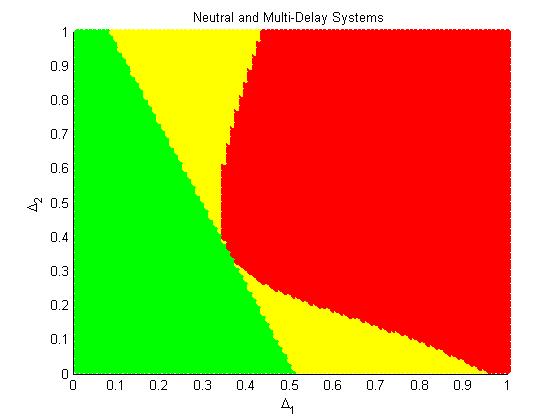}
         \caption{Neutral, $p = .7$}
     \end{subfigure}
     ~\hspace{-.01in}~
     \begin{subfigure}[b]{0.3\textwidth}
         \centering
         \includegraphics[width=\textwidth]{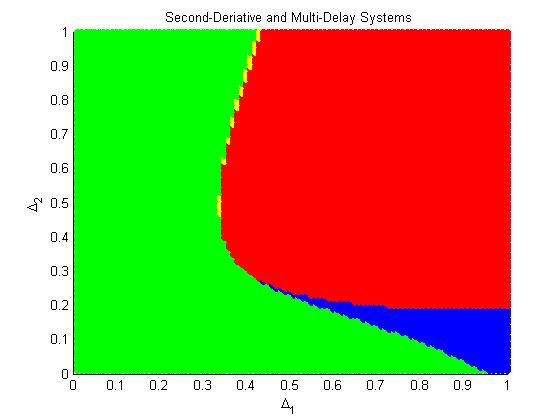}
          \centering
        \caption{$2^{nd}$-Derivative, $p = .7$}
     \end{subfigure}
     ~\hspace{-.01in}~
     \begin{subfigure}[b]{0.3\textwidth}
         \centering
         \includegraphics[width=\textwidth]{./Figures/multidelay/p70/multi.jpg}
         \caption{Multi-Delay, $p = .7$}
     \end{subfigure}
     ~\hspace{-.01in}~
\begin{subfigure}[b]{0.3\textwidth}
         \centering
       \includegraphics[width=\textwidth]{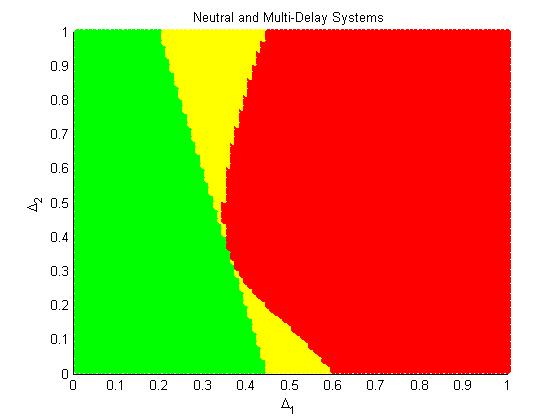}
         \caption{Neutral, $p = .8$}
     \end{subfigure}
     ~\hspace{-.01in}~
     \begin{subfigure}[b]{0.3\textwidth}
         \centering
         \includegraphics[width=\textwidth]{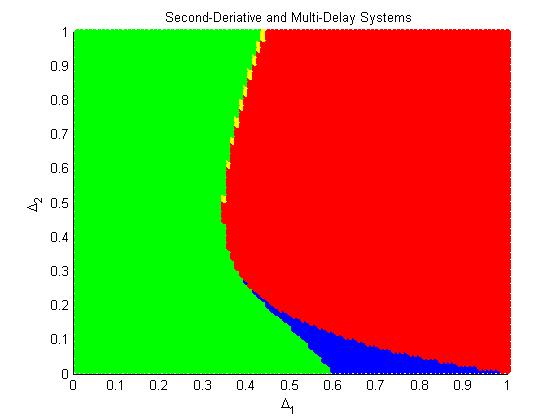}
          \centering
        \caption{$2^{nd}$-Derivative, $p = .8$}
     \end{subfigure}
     ~\hspace{-.01in}~
     \begin{subfigure}[b]{0.3\textwidth}
         \centering
         \includegraphics[width=\textwidth]{./Figures/multidelay/p80/multi.jpg}
         \caption{Multi-Delay, $p = .8$}
     \end{subfigure}
     ~\hspace{-.01in}~
\begin{subfigure}[b]{0.3\textwidth}
         \centering
       \includegraphics[width=\textwidth]{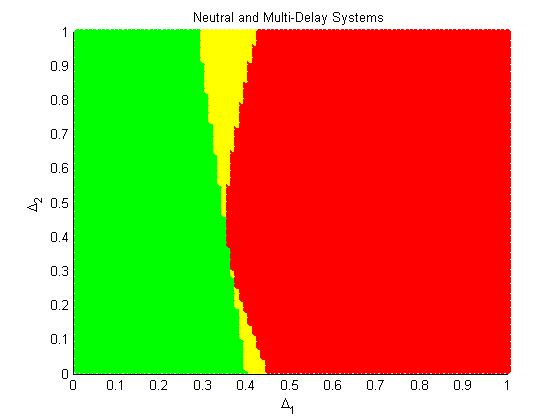}
         \caption{Neutral, $p = .9$}
     \end{subfigure}
     ~\hspace{-.01in}~
     \begin{subfigure}[b]{0.3\textwidth}
         \centering
         \includegraphics[width=\textwidth]{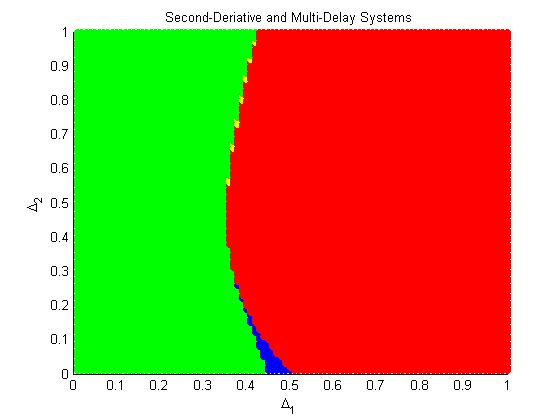}
          \centering
        \caption{$2^{nd}$-Derivative, $p = .9$}
     \end{subfigure}
     ~\hspace{-.01in}~
     \begin{subfigure}[b]{0.3\textwidth}
         \centering
         \includegraphics[width=\textwidth]{./Figures/multidelay/p90/multi.jpg}
         \caption{Multi-Delay, $p = .9$}
     \end{subfigure}
\begin{subfigure}[b]{0.3\textwidth}
         \centering
       \includegraphics[width=\textwidth]{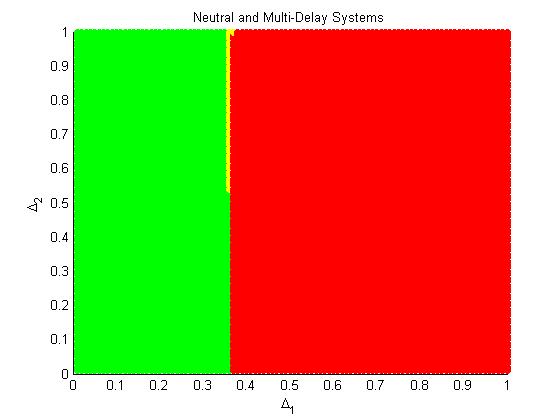}
         \caption{Neutral, $p = .99$}
     \end{subfigure}
     ~\hspace{-.01in}~
     \begin{subfigure}[b]{0.3\textwidth}
         \centering
         \includegraphics[width=\textwidth]{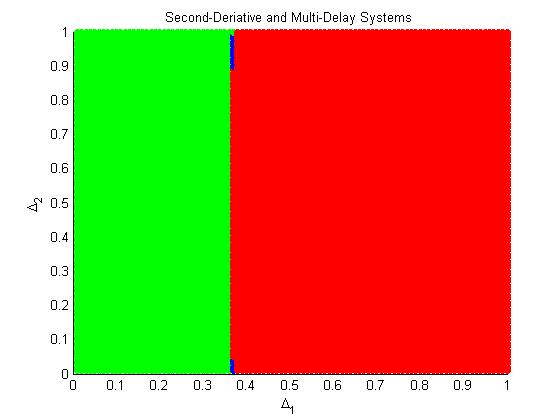}
          \centering
        \caption{$2^{nd}$-Derivative, $p = .99$}
     \end{subfigure}
     ~\hspace{-.01in}~
     \begin{subfigure}[b]{0.3\textwidth}
         \centering
         \includegraphics[width=\textwidth]{./Figures/multidelay/p99/multi.jpg}
         \caption{Multi-Delay, $p = .99$}
     \end{subfigure}
\end{figure}



\subsection{Comparing Symmetric Two-Delay and Three-Delay Systems}

After looking into the two-delay system numerically and observing where the change in stability occurs, it is interesting to consider how this would change with the addition of a third delay to the system. In particular, we would like to consider a two-delay system where each delay occurs with equal probability and then compare the stability of this system to that of a three-delay system where two of the delays are the same values as those in the two-delay system occuring with probability $p$, but the added delay is the midpoint of the other two delays and occurs with probability $1 - 2p$, as pictured in Figure \ref{picture_delays}.

\begin{figure}
     \begin{subfigure}[b]{0.49\textwidth}
         \centering
       \includegraphics[width=\textwidth]{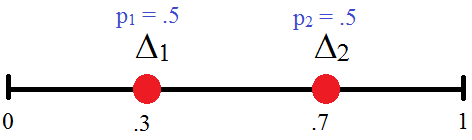}
     \end{subfigure}
     \begin{subfigure}[b]{0.49\textwidth}
         \centering
         \includegraphics[width=\textwidth]{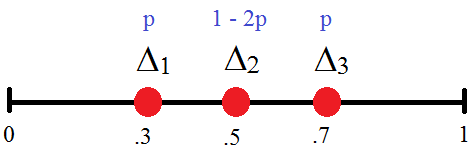}
          \centering
     \end{subfigure}
\caption{The delays for the two-delay system (pictured left) are $\Delta_1 = .3$ and $\Delta_2 = .7$, both occuring with equal probability. The delays for the three-delay system (pictured right) are $\Delta_1 = .3, \Delta_2 = .5$, and $\Delta_3 = .7$, where $\Delta_1$ and $\Delta_3$ each occur with probability $p$ and $\Delta_2$ occurs with probability $1 - 2p$. We note that $\mathbb{E}[\Delta] = .5$ for both of these example systems.}
\label{picture_delays}
\end{figure}

By this construction, $\mathbb{E}[\Delta]$ will be the same for all values of $p \in [0,1]$. The specific example that we will look at numerically has $\Delta_1 = .3$ and $\Delta_2 = .7$ in the two-delay system and $\Delta_1 = .3, \Delta_2 = .5, $ and $\Delta_3 = .7$ in the three-delay system.  We compare these two systems for various values of $p$ by checking whether each system is stable or unstable in each case. We collect this information in Table \ref{table_compare_2_3} and we observe that the two-delay system is more stable than the three-delay system in the cases we considered. To add a little intuition, we include some scatterplots in Figure \ref{3d_scatterplots_and_2d_scatterplot} showing the stability of the three-delay system for $p = .2$ and compare it so the stability scatterplot for the two-delay system. While these numerical examples are far from giving a conclusive test, it seems plausible that the two-delay system might always be more stable than the corresponding three-delay system with $\mathbb{E}[\Delta]$ remaining fixed.


\begin{table}[!ht]
\begin{center}
\begin{tabular}{| l | l | l | l | l | l |}
\hline\\ [-2.5ex]
$p$ & Two-Delay & Three-Delay \\
\hline
.1 &  Stable &  Unstable \\
\hline
.2 &  Stable &  Unstable\\
\hline
.3 &  Stable &  Unstable \\
\hline
.4 &   Stable & Unstable\\
\hline
.5 &  Stable & Stable\\
\hline
\end{tabular}
\end{center}
\caption{Comparison of the stability of the two-delay system with $\Delta_1 = .3$ and $\Delta_2 = .7$ with $p_1 = p_2 = .5$ to the three-delay system with $\Delta_1 = .3, \Delta_2 = .5, \Delta_3 = .7$ with $p_1 = p_3 = p$ and $p_2 = 1 - 2p$ for various values of $p$. We note that the three-delay system reduces to the two-delay system in the $p = .5$ case. }
\label{table_compare_2_3}
\end{table}


\begin{figure}
     \begin{subfigure}[b]{0.3\textwidth}
         \centering
       \includegraphics[width=\textwidth]{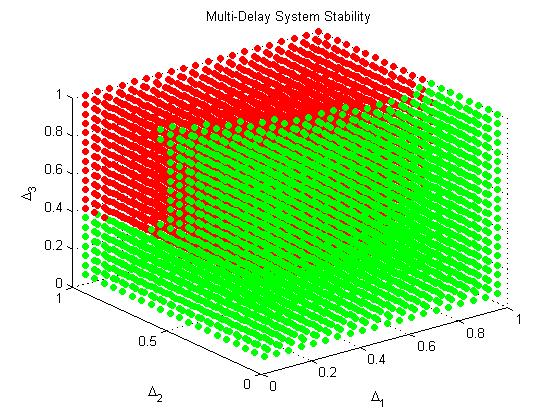}
     \end{subfigure}
     ~\hspace{-.01in}~
     \begin{subfigure}[b]{0.3\textwidth}
         \centering
         \includegraphics[width=\textwidth]{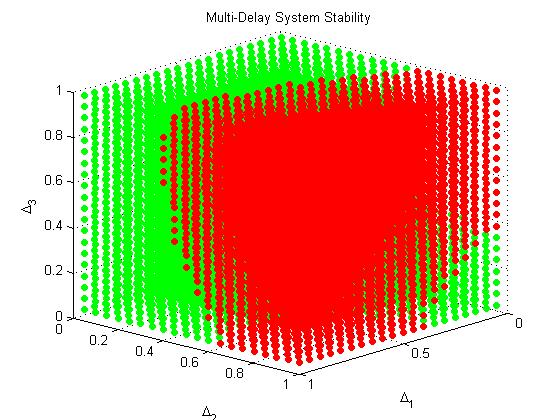}
          \centering
     \end{subfigure}
     ~\hspace{-.01in}~
     \begin{subfigure}[b]{0.3\textwidth}
         \centering
         \includegraphics[width=\textwidth]{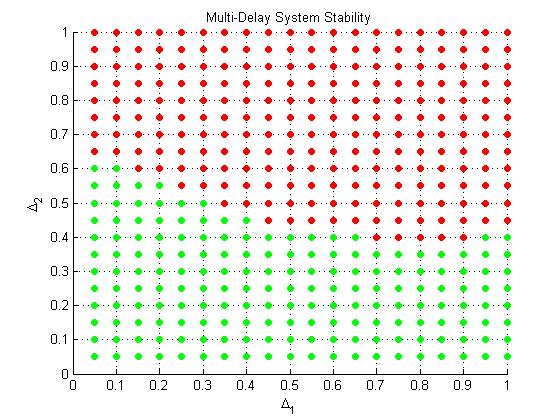}
     \end{subfigure}
\centering
     \begin{subfigure}[b]{0.3\textwidth}
         \centering
         \includegraphics[width=\textwidth]{./Figures/multidelay/p50/multi.jpg}
     \end{subfigure}
\caption{The top three scatterplots correspond to the three-delay system with $p_1 =  p_3 = .2$ and $p_2 = .6$, viewed from three different angles. The bottom scatterplot corresponds to the two-delay system with $p_1 = p_2 = .5$. Green points in the scatterplots correspond to points where the system is stable and red points in the scatterplots correspond to points where the system is unstable.}
\label{3d_scatterplots_and_2d_scatterplot}
\end{figure}


Motivated by these numerical results, we take a closer look at the two-delay and three-delay systems. The two-delay system is in the form $$\overset{\bullet}{u}(t) = \alpha_0 u(t) + C \left[ \frac{1}{2} u(t - \Delta_1) + \frac{1}{2} u(t - \Delta_3)  \right]$$ and the corresponding symmetric three-delay system is $$\overset{\bullet}{u}(t) = \alpha_0 u(t) + C \left[ p u(t - \Delta_1) + (1-2p) u\left(t - \frac{\Delta_1 + \Delta_3}{2}\right)+ p u(t - \Delta_3)  \right].$$ The corresponding characteristic polynomials for the two-delay and three-delay systems are 

\begin{align}
r &= \alpha_0 + C \left[ \frac{1}{2} e^{-r \Delta_1} + \frac{1}{2} e^{- \Delta_3} \right]\\
&= \alpha_0 + C \left[ \frac{1}{2} \left( e^{-r \frac{\Delta_1}{2}} - e^{-r \frac{\Delta_3}{2}}  \right)^2 + e^{-r \frac{\Delta_1}{2}}e^{-r \frac{\Delta_3}{2}}  \right]
\end{align}

\noindent and

\begin{align}
r &= \alpha_0 + C \left[  p e^{-r \Delta_1} + (1-2p)e^{-r \left( \frac{\Delta_1 + \Delta_3}{2} \right) } + p e^{-r \Delta_3}  \right]\\
&= \alpha_0 + C \left[ p \left( e^{-r \frac{\Delta_1}{2}} - e^{-r \frac{\Delta_3}{2}}  \right)^2 + e^{-r \frac{\Delta_1}{2}}e^{-r \frac{\Delta_3}{2}}  \right]
\end{align}

\noindent respectively. Let $r = a + i b$ for $a,b \in \mathbb{R}$. We now aim to split each of the characteristic equations into the real and imaginary parts of $r$ and then we will specifically analyze the equations corresponding to the real part of $r$ as this is the part that determines the stability of the system. Below we have the equations corresponding to the real part of $r$ for the two-delay and three-delay systems, respectively. 

\begin{align}
a &= \alpha_0 + C \left[ \frac{1}{2} \left( e^{-a \Delta_1} + e^{-a \Delta_3}  \right) \right]\\
&:= \alpha_0 + C  f_2(a)
\end{align}

\begin{align}
a &= \alpha_0 + C \left[ p \left( e^{-a \Delta_1} + e^{-a \Delta_3}  \right) + (1 - 2p)e^{-a \left(  \frac{\Delta_1 + \Delta_3}{2} \right) } \cos \left( b \left( \frac{\Delta_1 + \Delta_3}{2} \right)  \right)  \right]\\
&:= \alpha_0 + C f_3(a)
\end{align}

\noindent where 

\begin{align}
f_2(a) &= \frac{1}{2} \left( e^{-a \Delta_1} + e^{-a \Delta_3}  \right)\\
f_3(a) &=  p \left( e^{-a \Delta_1} + e^{-a \Delta_3}  \right) + (1 - 2p)e^{-a \left(  \frac{\Delta_1 + \Delta_3}{2} \right) } \cos \left( b \left( \frac{\Delta_1 + \Delta_3}{2} \right)  \right).
\end{align}

\noindent Since we are assuming that $C < 0$, if we can show that $f_2(a) \geq f_3(a)$ for all $a \in \mathbb{R}$, then we will have that the real parts of the roots of the characteristic polynomial for the two-delay system will be less than or equal to those corresponding to the three-delay system for all $r \in \mathbb{C}$. This would suggest that the two-delay system is more stable the the three-delay system. Thus, it remains to show that $f_2(a) \geq f_3(a)$. First, we observe that $f_3(a)$ has the following upper bound $$f_3(a) \leq  p \left( e^{-a \Delta_1} + e^{-a \Delta_3}  \right) + (1 - 2p)e^{-a \left(  \frac{\Delta_1 + \Delta_3}{2} \right) } := g_3(a).$$ Thus, if we can show that $f_2(a) - g_3(a) \geq 0$, then we are done.

\begin{align}
f_2(a) - g_3(a) &=  \left( \frac{1}{2} - p \right) \left( e^{-a \Delta_1} + e^{-a \Delta_3}  \right) - (1 - 2p)e^{-a \left(  \frac{\Delta_1 + \Delta_3}{2} \right) }\\
&=  \left( 1 - 2p \right)  \left[  \frac{1}{2}\left( e^{-a \Delta_1} + e^{-a \Delta_3}  \right) - e^{-a \left(  \frac{\Delta_1 + \Delta_3}{2} \right) } \right]\\
&\geq 0
\end{align}

\noindent because $$\frac{1}{2}\left( e^{-a \Delta_1} + e^{-a \Delta_3}  \right) \geq e^{-a \left(  \frac{\Delta_1 + \Delta_3}{2} \right) }$$ by Jensen's inequality. This result essentially shows that the two-delay system is more stable than the three-delay system. 


\section{Potential Insight from the Lambert W Function}
\label{lambertwsection}

The Lambert W function can be used to solve linear constant-coefficient delay differential equations analytically. First consider the constant-delay differential equation $$\overset{\bullet}{u}(t) = \alpha_0 u(t) + A_0 u(t - \Delta^*).$$ We note that our neutral approximation reduces to a DDE in this form when $\Delta^* = \mathbb{E}[\Delta^*]$. By looking for solutions in the form $u(t) = e^{rt}$, we obtain the characteristic equation corresponding to this DDE which is $$r = \alpha_0 + A_0 e^{-r \Delta^*}$$ and we can rearrange this equation to be in the form $$ \Delta^*(r - \alpha_0) e^{\Delta^* (r - \alpha_0)} = \Delta^*  A_0 e^{-\alpha_0 \Delta^* }$$ so that it can be directly solved using the Lambert W function to obtain $$r = \frac{ W_k \left( \Delta^* A_0 e^{-\alpha_0 \Delta^*}  \right)}{\Delta^*} + \alpha_0$$ where $W_k(\cdot)$ denotes the $k$-th branch of the Lambert W function. This gives us a general solution in the form $$u(t) = \sum_{k = - \infty}^{\infty} c_k \exp \left( \left[ \frac{ W_k \left( \Delta^* A_0 e^{-\alpha_0 \Delta^*}  \right)}{\Delta^*} + \alpha_0  \right] t \right).$$

Now consider a DDE in the form of our neutral approximation (but with $\Delta^* \neq \mathbb{E}[\Delta]$ so that $A_1 \neq 0$, otherwise the analysis reduces to the case we just considered above). $$\overset{\bullet}{u}(t) = \alpha_0 u(t) + A_0 u(t - \Delta^*) + A_1 \overset{\bullet}{u}(t - \Delta^*)$$ The characteristic equation corresponding to this DDE is 

\begin{align*}
r &= \alpha_0 + A_0 e^{- r \Delta^*} + r A_1 e^{-r \Delta^*}\\
&= \alpha_0 + e^{-r \Delta^*}(A_0 + r A_1)
\end{align*}

\noindent which can be rewritten to be in the form $$ e^{r \Delta^*} = \frac{r A_1 + A_0}{r - \alpha_0}.$$ This equation can not be explicitly solved for $r$ with the Lambert W function and instead we must use a generalization of the Lambert W function, as was introduced in \cite{mezHo2017generalization}. Taking a look at a DDE in the form of our second-derivative approximation $$\overset{\bullet}{u}(t) = \alpha_0 u(t) + A_0 u(t - \Delta^*) + A_1 \overset{\bullet}{u}(t  - \Delta^*) + A_2 \overset{\bullet \bullet}{u}(t - \Delta^*)$$ we see that the corresponding characteristic equation is $$r = \alpha_0 + e^{- r \Delta^*} (A_0 + r A_1 + r^2 A_2)$$ which can be rewritten to be in the form $$e^{r \Delta^*} = \frac{r^2 A_2 + r A_1 + A_0}{r - \alpha_0}.$$ Assuming that $\alpha_0$ is not a root of the polynomial $r^2 A_2 + r A_1 + A_0$, then this characteristic equation also cannot be solved explicitly with only the Lambert W function and the aforementioned generalization to the Lambert W function could instead be used to solve for $r$ explicitly.  

In general, suppose that we considered a single-delay approximation that kept the first $n$ derivatives, yielding a DDE in the form $$\overset{\bullet}{u}(t) = \alpha_0 u(t) + \sum_{k=0}^{n} A_k u^{(k)}(t - \Delta^*).$$ The corresponding characteristic equation is $$r = \alpha_0 + e^{-r \Delta^*} \sum_{k=0}^{n} A_k r^k $$ which can be rewritten to be in the form $$e^{r \Delta^*} = \frac{\sum_{k=0}^{n} A_k r^k}{r - \alpha_0}.$$ If the polynomial $\sum_{k=0}^{n} A_k r^k$ has roots $t_1, ..., t_n$, then we can rewrite this as $$ \frac{r - \alpha_0}{(r - t_1) \cdots (r - t_n)} e^{r \Delta^*} = 1.$$ Let $r = \frac{z}{\Delta^*}$ and the equation becomes $$\frac{z - \Delta^* \alpha_0}{(z - \Delta^* t_1) \cdots (z - \Delta^* t_n)} e^z = \frac{1}{(\Delta^*)^{n-1}}$$ so an analysis of the multivalued inverse function of $$f(z) = \frac{z - a_0}{(z - a_1) \cdots (z - a_n)} e^z$$ for some constants $a_0, ..., a_n$ would be sufficient for better understanding the solutions and stability of DDEs in this form. Indeed, the inverse of such a function would generalize the Lambert W function which is the multivalued inverse of  $$f(z) = ze^z.$$

\section{Conclusion and Future Research}
\label{conclusion}

In this paper we introduced two single-delay approximations to a system with multiple constant delays with the goal of approximating where the change in stability occurs in the multi-delay system. Our single-delay approximations were obtained by Taylor expanding delayed terms in the multi-delay DDE about some constant $\Delta^*$ and then truncating higher-order terms. When keeping only the first derivative term from the Taylor expansion, we obtained our neutral approximation to the multi-delay system and derived a formula for the critical delay of this neutral DDE which we used to approximate where the change in stability occurs in the multi-delay system. We observed that making the choice $\Delta^* = \mathbb{E}[\Delta]$ in our neutral approximation gave better results than other choices of $\Delta^*$ and we noted that this has to do with the coefficient of the leading-order error term from the Taylor expansion being minimized for this choice of $\Delta^*$. 

Unsatisfied with the performance of the neutral approximation in the two-delay setting near $p = \frac{1}{2}$, we considered keeping the second-derivative term from the Taylor expansion to get our other approximation to the multi-delay system. After deriving the critical delay for our second-derivative approximation, we saw that it performed quite well in the two-delay setting, approximating the change in stability of the multi-delay system with greater than $94 \%$ accuracy in all cases we considered with about $98 \%$ accuracy on average over various values of $p$ when sampling $(\Delta_1, \Delta_2)$ pairs from the unit square, as seen in Table \ref{Table_Accuracy_2nd}. 

We took a closer look at a two-delay model and found values of $p$ that maximized the coefficients of the leading-order error terms in the Taylor expansions for each approximation which helped explain why each approximation did not perform as well near certain values of $p$. We also briefly compared the stability of two-delay and three-delay systems with the same value of $\mathbb{E}[\Delta]$ numerically and we indeed saw analytically that the two-delay system was more stable the the three-delay system with appropriate symmetric probability distributions for the delays in each case. It would be interesting to see if this result could be generalized to asymmetric probability distributions. Another natural extension to our work would be to do a more in-depth analysis of a system with more than two constant delays to better understand how our approximations perform in more complicated settings. 

One potential extension to our work could be to do an analysis of approximations that include more terms from the Taylor expansion we used, such as including the delayed third-derivative term in the approximation. In this case, the cubic formula can be used to find an expression for the critical delay, though we have not analyzed this case in full detail. Another possible extension for future research could be to consider similar approximations to more general multi-delay systems with higher-order derivative terms.


\section{Appendix} 

\subsection{Proof of Theorem \ref{fluidlimit}}

Before we prove Theorem \ref{fluidlimit}, we state a useful lemma from \cite{kurtz1978strong} that will be used in the proof. 

\begin{lemma}
\label{brownian_lemma}
A standard Poisson process $\{ \Pi(t) \}_{t \geq 0 }$ can be realized on the same probability space as a standard brownian motion $\{ W(t) \}_{t \geq 0}$ in such a way that the almost surely finite random variable $$Z \equiv \sup_{t \geq 0} \frac{|\Pi(t) - t - W(t)|}{\log(2 \vee t )}$$ has finite moment generating function in a neighborhood of the origin and in particular finite mean.
\end{lemma}

\noindent \textbf{Proof of Theorem \ref{fluidlimit}.}

\begin{proof}

We will proceed with a proof very similar to that found in \cite{pender2020stochastic}. Recall that the scaled queueing process is defined as follows.

 \begin{eqnarray}
Q^{\eta}_i(t) &=& Q^{\eta}_i(0) +  \frac{1}{\eta}\Pi^a_{i} \left( \eta \int^{t}_{0} \frac{  \lambda \cdot \exp(- \theta \left( \sum^{m}_{k=1} p_k Q^{\eta}_i(s-\Delta_k) \right))}{ \sum^{N}_{j=1} \exp(- \theta \left( \sum^{m}_{k=1} p_k Q^{\eta}_j(s-\Delta_k) \right) ) } ds \right) \nonumber \\
&&-  \frac{1}{\eta}\Pi^d_{i} \left(\eta \int^{t}_{0} \mu Q^{\eta}_i(s) ds \right).
\end{eqnarray}

From Lemma \ref{brownian_lemma}, we can deduce that for every $t > 0$ and for sufficiently large $\eta$, we have that $$\frac{\Pi(\eta t)}{\eta} = \frac{W(\eta t)}{\eta} + t + O \left( \frac{\log(\eta)}{\eta}  \right).$$

We take the difference of the scaled queue length process and the fluid limit in preparation for finding an upper bound to it.

\begin{align*}
Q_i^{\eta}(t) - q_i(t) &= Q_i^{\eta}(0) - q_i(0) + \frac{1}{\eta} \Pi_i^{a} \left( \eta \int_{0}^{t} \frac{\lambda \cdot \exp \left( -\theta \sum_{k=1}^{m} Q_i^{\eta}(s - \Delta_k) \right) }{ \sum_{j=1}^{N} \exp \left( -\theta \sum_{k=1}^{m} Q_j^{\eta}(s - \Delta_k) \right) } ds  \right) \\
&- \int_{0}^{t} \frac{\lambda \cdot \exp \left( -\theta \sum_{k=1}^{m} q_i^{\eta}(s - \Delta_k) \right) }{ \sum_{j=1}^{N} \exp \left( -\theta \sum_{k=1}^{m} q_j^{\eta}(s - \Delta_k) \right) } ds\\
&- \frac{1}{\eta} \Pi_i^{d} \left( \eta \int_{0}^{t} \mu Q_i^{\eta}(s) ds  \right) + \int_{0}^{t} \mu q_i(s) ds\\
&= Q_i^{\eta}(0) - q_i(0)\\
&+  \frac{1}{\eta} \Pi_i^{a} \left( \eta \int_{0}^{t} \frac{\lambda \cdot \exp \left( -\theta \sum_{k=1}^{m} Q_i^{\eta}(s - \Delta_k) \right) }{ \sum_{j=1}^{N} \exp \left( -\theta \sum_{k=1}^{m} Q_j^{\eta}(s - \Delta_k) \right) } ds  \right) \\
&+  \int_{0}^{t} \frac{\lambda \cdot \exp \left( -\theta \sum_{k=1}^{m} Q_i^{\eta}(s - \Delta_k) \right) }{ \sum_{j=1}^{N} \exp \left( -\theta \sum_{k=1}^{m} Q_j^{\eta}(s - \Delta_k) \right) } ds  - \int_{0}^{t} \frac{\lambda \cdot \exp \left( -\theta \sum_{k=1}^{m} Q_i^{\eta}(s - \Delta_k) \right) }{ \sum_{j=1}^{N} \exp \left( -\theta \sum_{k=1}^{m} Q_j^{\eta}(s - \Delta_k) \right)  } ds\\
&- \int_{0}^{t} \frac{\lambda \cdot \exp \left( -\theta \sum_{k=1}^{m} q_i^{\eta}(s - \Delta_k) \right) }{ \sum_{j=1}^{N} \exp \left( -\theta \sum_{k=1}^{m} q_j^{\eta}(s - \Delta_k) \right) } ds\\
&- \frac{1}{\eta} \Pi_i^{d} \left( \eta \int_{0}^{t} \mu Q_i^{\eta}(s) ds  \right) + \int_{0}^{t} \mu Q_i^{\eta}(s) ds - \int_0^{t} \mu Q_i^{\eta}(s) ds + \int_{0}^{t} \mu q_i(s) ds
\end{align*}

\noindent We added and subtracted the same term a couple times to put the expression in a better form to apply the result from Lemma \ref{brownian_lemma}, which tells us that 

\begin{align*}
Q_i^{\eta}(t) &= Q_i^{\eta}(0) +  \frac{1}{\eta} W_i^{a} \left( \eta \int_{0}^{t} \frac{\lambda \cdot \exp \left( -\theta \sum_{k=1}^{m} Q_i^{\eta}(s - \Delta_k) \right) }{ \sum_{j=1}^{N} \exp \left( -\theta \sum_{k=1}^{m} Q_j^{\eta}(s - \Delta_k) \right) } ds  \right) \\
&+  \int_{0}^{t} \frac{\lambda \cdot \exp \left( -\theta \sum_{k=1}^{m} Q_i^{\eta}(s - \Delta_k) \right) }{ \sum_{j=1}^{N} \exp \left( -\theta \sum_{k=1}^{m} Q_j^{\eta}(s - \Delta_k) \right) } ds \\
&- \frac{1}{\eta} W_i^{d} \left( \eta \int_{0}^{t} \mu Q_i^{\eta}(s) ds  \right) - \int_{0}^{t} \mu Q_i^{\eta}(s) ds + O \left( \frac{\log(\eta)}{\eta}  \right)\\
&= Q_i^{\eta}(0) +  \frac{1}{\sqrt{\eta}} W_i^{a} \left(  \int_{0}^{t} \frac{\lambda \cdot \exp \left( -\theta \sum_{k=1}^{m} Q_i^{\eta}(s - \Delta_k) \right) }{ \sum_{j=1}^{N} \exp \left( -\theta \sum_{k=1}^{m} Q_j^{\eta}(s - \Delta_k) \right) } ds  \right) \\
&+  \int_{0}^{t} \frac{\lambda \cdot \exp \left( -\theta \sum_{k=1}^{m} Q_i^{\eta}(s - \Delta_k) \right) }{ \sum_{j=1}^{N} \exp \left( -\theta \sum_{k=1}^{m} Q_j^{\eta}(s - \Delta_k) \right) } ds \\
&- \frac{1}{\sqrt{\eta}} W_i^{d} \left( \int_{0}^{t} \mu Q_i^{\eta}(s) ds  \right) - \int_{0}^{t} \mu Q_i^{\eta}(s) ds + O \left( \frac{\log(\eta)}{\eta}  \right)
\end{align*}

\noindent where $W_i^a$ and $W_i^d$ are standard Brownian motions and we used the scaling property of Brownian motion in the second equality.

Taking the absolute value and using the result from Lemma \ref{brownian_lemma}, we obtain the following upper bound on the difference.

\begin{align*}
|Q_i^{\eta}(t) - q_i(t) | &\leq  \left| Q_i^{\eta}(0) - q_i(0) \right|\\
&+ \left| \frac{1}{\sqrt{\eta}} W_i^{a} \left(  \int_{0}^{t} \frac{\lambda \cdot \exp \left( -\theta \sum_{k=1}^{m} Q_i^{\eta}(s - \Delta_k) \right) }{ \sum_{j=1}^{N} \exp \left( -\theta \sum_{k=1}^{m} Q_j^{\eta}(s - \Delta_k) \right) } ds  \right)  \right|\\
&+ \left| \frac{1}{\sqrt{\eta}} W_i^{d} \left( \int_{0}^{t} \mu Q_i^{\eta}(s) ds  \right)  \right|\\
&+ \left|  \int_{0}^{t} \frac{\lambda \cdot \exp \left( -\theta \sum_{k=1}^{m} Q_i^{\eta}(s - \Delta_k) \right) }{ \sum_{j=1}^{N} \exp \left( -\theta \sum_{k=1}^{m} Q_j^{\eta}(s - \Delta_k) \right) } ds  - \int_{0}^{t} \frac{\lambda \cdot \exp \left( -\theta \sum_{k=1}^{m} q_i^{\eta}(s - \Delta_k) \right) }{ \sum_{j=1}^{N} \exp \left( -\theta \sum_{k=1}^{m} q_j^{\eta}(s - \Delta_k) \right) } ds  \right|\\
&+ \left| \int_0^{t} \mu Q_i^{\eta}(s) ds - \int_{0}^{t} \mu q_i(s) ds  \right| + O \left( \frac{\log(\eta)}{\eta}  \right)
\end{align*}

\noindent We now want to show that, for all $T > 0$,  $$\lim_{\eta \to \infty} \sup_{t \leq T} \left| \frac{1}{\sqrt{\eta}} W_i^{a} \left(  \int_{0}^{t} \frac{\lambda \cdot \exp \left( -\theta \sum_{k=1}^{m} Q_i^{\eta}(s - \Delta_k) \right) }{ \sum_{j=1}^{N} \exp \left( -\theta \sum_{k=1}^{m} Q_j^{\eta}(s - \Delta_k) \right) } ds  \right)  \right| = 0$$ and $$\lim_{\eta \to \infty} \sup_{t \leq T} \left| \frac{1}{\sqrt{\eta}} W_i^{d} \left( \int_{0}^{t} \mu Q_i^{\eta}(s) ds  \right)  \right| = 0.$$ We have 

\begin{align*}
\lim_{\eta \to \infty} \sup_{t \leq T} \left| \frac{1}{\sqrt{\eta}} W_i^{a} \left(  \int_{0}^{t} \frac{\lambda \cdot \exp \left( -\theta \sum_{k=1}^{m} Q_i^{\eta}(s - \Delta_k) \right) }{ \sum_{j=1}^{N} \exp \left( -\theta \sum_{k=1}^{m} Q_j^{\eta}(s - \Delta_k) \right) } ds  \right)  \right| &\leq \lim_{\eta \to \infty} \sup_{t \leq T} \left| \frac{1}{\sqrt{\eta}} W_i^{a} (\lambda \cdot t)  \right|\\
&= \lim_{\eta \to \infty} \sup_{t \leq T} \left| W_i^{a} \left( \frac{\lambda \cdot t}{\eta}  \right)  \right|\\
&= 0
\end{align*}

\noindent where we used the fact that the multinomial logit function is bounded above by $1$, Brownian scaling, and the fact that $W(0) = 0$. Similarly, for the next term we have

\begin{align*}
\lim_{\eta \to \infty} \sup_{t \leq T} \left| \frac{1}{\sqrt{\eta}} W_i^{d} \left( \int_{0}^{t} \mu Q_i^{\eta}(s) ds  \right)  \right| &\leq \lim_{\eta \to \infty} \sup_{t \leq T} \left| \frac{1}{\sqrt{\eta}} W_i^{d} \left( \mu \cdot t \cdot (Q_i^{\eta}(0) + \lambda )   \right)  \right|\\
&= \lim_{n \to \infty} \sup_{t \leq T} \left|  W_i^{d} \left( \frac{1}{\eta} \cdot \mu \cdot t \cdot (Q_i^{\eta}(0) + \lambda )   \right)  \right|\\
&= 0.
\end{align*}

\noindent Thus, we have that for every $\epsilon > 0$, there exists some $\eta^*$ such that for every $\eta \geq \eta^*$ we have that $$\lim_{\eta \to \infty} \sup_{t \leq T} \left| \frac{1}{\sqrt{\eta}} W_i^{a} \left(  \int_{0}^{t} \frac{\lambda \cdot \exp \left( -\theta \sum_{k=1}^{m} Q_i^{\eta}(s - \Delta_k) \right) }{ \sum_{j=1}^{N} \exp \left( -\theta \sum_{k=1}^{m} Q_j^{\eta}(s - \Delta_k) \right) } ds  \right)  \right| < \frac{\epsilon}{5},$$

$$\lim_{\eta \to \infty} \sup_{t \leq T} \left| \frac{1}{\sqrt{\eta}} W_i^{d} \left( \int_{0}^{t} \mu Q_i^{\eta}(s) ds  \right)  \right| < \frac{\epsilon}{5},$$

$$\left| Q_i^{\eta}(0) - q_i(0)  \right| < \frac{\epsilon}{5},$$ and $$O \left( \frac{\log \eta}{\eta}  \right) < \frac{\epsilon}{5}$$ so that 

\begin{align*}
\left| Q_i^{\eta}(t) - q_i(t)  \right| &\leq \left|  \int_{0}^{t} \frac{\lambda \cdot \exp \left( -\theta \sum_{k=1}^{m} Q_i^{\eta}(s - \Delta_k) \right) }{ \sum_{j=1}^{N} \exp \left( -\theta \sum_{k=1}^{m} Q_j^{\eta}(s - \Delta_k) \right) } ds  - \int_{0}^{t} \frac{\lambda \cdot \exp \left( -\theta \sum_{k=1}^{m} q_i^{\eta}(s - \Delta_k) \right) }{ \sum_{j=1}^{N} \exp \left( -\theta \sum_{k=1}^{m} q_j^{\eta}(s - \Delta_k) \right) } ds  \right|\\
&+ \left| \int_0^{t} \mu Q_i^{\eta}(s) ds - \int_{0}^{t} \mu q_i(s) ds  \right| + \frac{4 \epsilon}{5}\\
&\leq  \int_{0}^{t} \left|  \frac{\lambda \cdot \exp \left( -\theta \sum_{k=1}^{m} Q_i^{\eta}(s - \Delta_k) \right) }{ \sum_{j=1}^{N} \exp \left( -\theta \sum_{k=1}^{m} Q_j^{\eta}(s - \Delta_k) \right) }   -  \frac{\lambda \cdot \exp \left( -\theta \sum_{k=1}^{m} q_i^{\eta}(s - \Delta_k) \right) }{ \sum_{j=1}^{N} \exp \left( -\theta \sum_{k=1}^{m} q_j^{\eta}(s - \Delta_k) \right) } \right| ds  \\
&+  \int_0^{t} \left| \mu Q_i^{\eta}(s)  - \int_{0}^{t} \mu q_i(s) \right| ds   + \frac{ 4 \epsilon}{5}
\end{align*}

\noindent Using the fact that the multinomial logit probability function and the departure function are both differentiable with bounded first derivatives, there exists some constant $C$ such that

\begin{align*}
\left| Q_i^{\eta}(t) - q_i(t) \right| &\leq C \int_0^{t} \sup_{- \max_k(\Delta_k) \leq r \leq s} \left|  Q_i^{\eta}(r) - q_i(r)  \right| ds + \frac{4 \epsilon}{5}\\
&\leq C \cdot \left( \int_0^t \sup_{0 \leq r \leq s} \left| Q_i^{\eta}(r) - q_i(r)   \right| ds  + \sup_{- \max_k(\Delta_k) \leq r \leq 0} \left| Q_i^{\eta}(r) - q_i(r)  \right| \cdot t  \right) + \frac{4 \epsilon}{5}.
\end{align*}

\noindent Since we assumed that $Q_i^{\eta}(s) \to \phi_i(s)$ almost surely for all $s \in [- \max_k(\Delta_k), 0]$ and for each $1 \leq i \leq N$ and that $q_i(s) = \phi_i(s)$ for all $s \in [- \max_k(\Delta_k), 0]$ and $1 \leq i \leq N$, then if $\eta^*$ is sufficiently large so that we also have that $$\sup_{- \max_k(\Delta_k) \leq r \leq 0} < \frac{\epsilon}{5 C t},$$ then we have the following bound almost surely $$\left| Q_i^{\eta}(t) - q_i(t)  \right| \leq C \int_0^t \sup_{0 \leq r \leq s} \left| Q_i^{\eta}(r) - q_i(r)  \right|  ds + \epsilon. $$ Using Gronwall's lemma in \cite{hale}, we have that $$\sup_{0 \leq t \leq T} \left| Q_i^{\eta}(t) - q_i(t)  \right| \leq \epsilon \cdot e^{CT}$$ and since $\epsilon > 0$ is arbitrarily small, we can let it approach zero. This proves the fluid limit.

\end{proof}

\subsection{Scatterplots from Larger Regions}

Earlier we made note of the arbitrariness of the definition of accuracy being used. In this section, we show some examples of scatterplots obtained from considering points $(\Delta_1, \Delta_2)$ in regions other than the unit square. In Figure \ref{Fig16} and Figure \ref{Fig17} we show scatterplots from sampling points from a two-by-two square as well as those sampling from a five-by-five square, respectively. It is easy to see that as the square that we sample $(\Delta_1, \Delta_2)$ pairs from gets larger, the percentage of points in the neutral approximation scatterplot that match those in the corresponding multi-delay scatterplot increases whereas the percentage of points in the second-derivative approximation scatterplot that match those in the multi-delay scatterplot decreases. While this may initially make the second-derivative approximation appear less promising than it did in the unit-square setting, it is worth noting that the multi-delay scatterplot seems to have a relatively simple structure for large values of $\Delta_1$ or $\Delta_2$ as the red-green border appears to be easy to approximate with either a horizontal or vertical line for sufficiently large delays. Thus, the fact that much of the nonlinearity of the red-green border occurs in the unit square is reassuring regarding the meaningfulness of the accuracy statistics collected earlier. 


\begin{figure}[!htb]
     \begin{subfigure}[b]{0.3\textwidth}
         \centering
       \includegraphics[width=\textwidth]{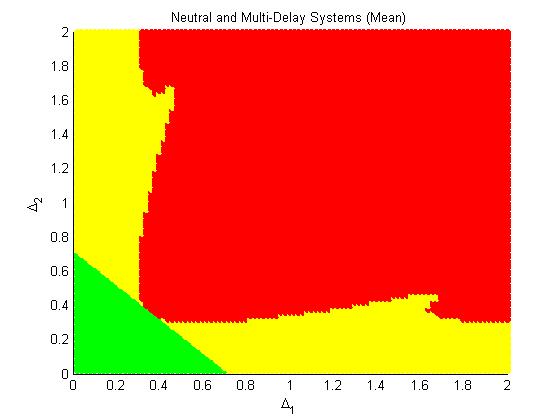}
         \caption{Neutral, $p = .5$}
     \end{subfigure}
     ~\hspace{-.01in}~
     \begin{subfigure}[b]{0.3\textwidth}
         \centering
         \includegraphics[width=\textwidth]{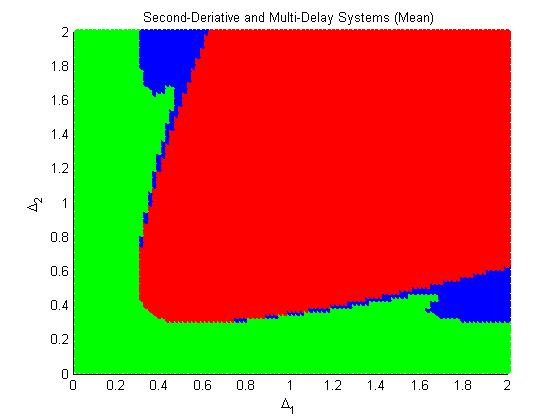}
          \centering
        \caption{$2^{nd}$-Derivative, $p = .5$}
     \end{subfigure}
     ~\hspace{-.01in}~
     \begin{subfigure}[b]{0.3\textwidth}
         \centering
         \includegraphics[width=\textwidth]{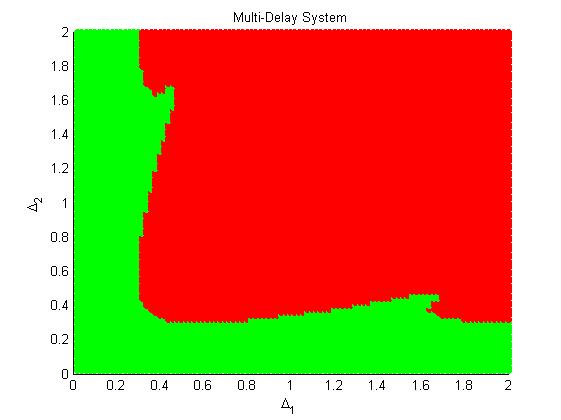}
         \caption{Multi-Delay, $p = .5$}
     \end{subfigure}
\caption{Scatterplots for the neutral and second-derivative approximations with $\Delta^* = \mathbb{E}[\Delta]$ as well as a scatterplot for the actual two-delay system, all sampling $(\Delta_1, \Delta_2)$ pairs from a \textbf{two-by-two} square and with $p = .5, C = -5, \alpha_0 = -1$. The analogously-defined accuracy for the neutral approximation is about $74.62 \%$ and for the second-derivative approximation it is about $94.54 \%$.}
\label{Fig16}
\end{figure}
\begin{figure}[!htb]
     \begin{subfigure}[b]{0.3\textwidth}
         \centering
       \includegraphics[width=\textwidth]{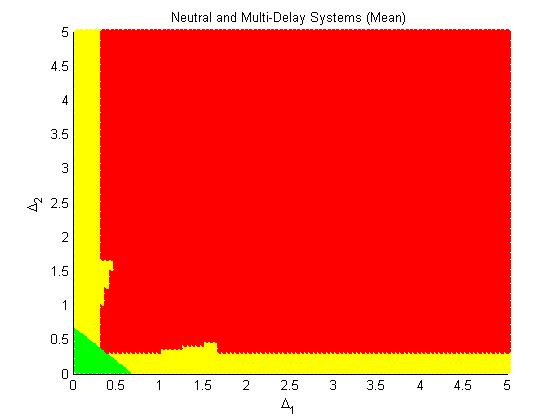}
         \caption{Neutral, $p = .5$}
     \end{subfigure}
     ~\hspace{-.01in}~
     \begin{subfigure}[b]{0.3\textwidth}
         \centering
         \includegraphics[width=\textwidth]{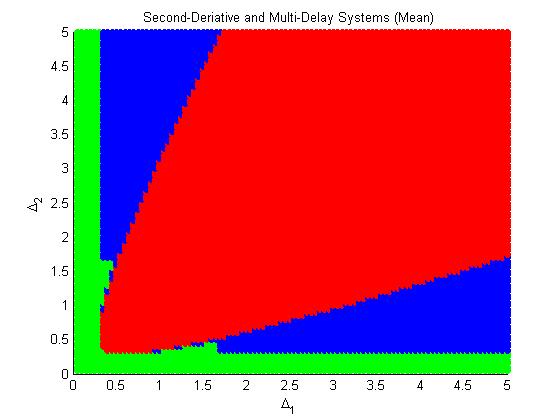}
          \centering
        \caption{$2^{nd}$-Derivative, $p = .5$}
     \end{subfigure}
     ~\hspace{-.01in}~
     \begin{subfigure}[b]{0.3\textwidth}
         \centering
         \includegraphics[width=\textwidth]{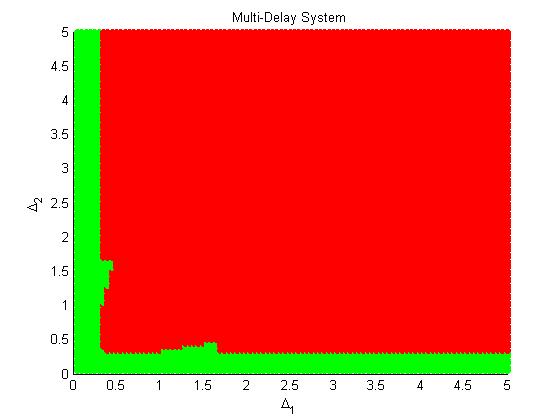}
         \caption{Multi-Delay, $p = .5$}
     \end{subfigure}
\caption{Scatterplots for the neutral and second-derivative approximations with $\Delta^* = \mathbb{E}[\Delta]$ as well as a scatterplot for the actual two-delay system, all sampling $(\Delta_1, \Delta_2)$ pairs from a \textbf{five-by-five square} and with $p = .5, C = -5, \alpha_0 = -1$. The analogously-defined accuracy for the neutral approximation is about $88.78 \%$ and for the second-derivative approximation it is about $79.10 \%$.}
\label{Fig17}
\end{figure}

\bibliography{neutral_arxiv}

\bibliographystyle{plainnat}

\end{document}